\documentclass[11pt]{amsart}
\usepackage{amsmath}
\usepackage{amsfonts}
\usepackage{amsthm}
\usepackage{amssymb}
\usepackage{mathrsfs}
\usepackage{latexsym}
\usepackage{verbatim}
\usepackage{diagrams}
\usepackage{a4wide}
\usepackage{enumerate}
\usepackage{hyperref}
\usepackage{bbding}
\usepackage{pifont}
\usepackage[all]{xy}
\usepackage{rotating}
\usepackage{colortbl}

\def\fff{{}\kern-0.1em{f}}
\def\GammaG{\Phi}

\def\Gal{\mathrm{Gal}}

\def\RCl{\mathrm{RCl}}
\def\p{\mathfrak{p}}
\def\pbar{\overline{\p}}
\def\alphabar{\overline{\alpha}}
\def\y{\hat{x}}
\def\Tor{\mathrm{Tor}}
\def\Sp{\mathrm{Sp}}
\def\q{\mathfrak{q}}
\def\Hc{H^{\mathrm{cont}}}
\def\Hs{H^{\mathrm{stab}}}
\def\Qbar{\overline{\Q}}

\def\Cl{\mathrm{Cl}}

\def\Yoon{Y_{\infty}}

\def\Ymn{Y_{m}}

\def\EE{\mathbf{E}}

\def\R{\mathbf{R}}

\def\OL{\mathscr{O}}

\def\LL{M}
\def\MM{\mathscr{L}}
\def\Hom{\mathrm{Hom}}
\def\Spec{\mathrm{Spec}}

\newarrow{Equals}{=}{=}{=}{=}{=}
\newarrow{Spectral}{=}{=}{=}{=}{>}
\newarrow{dotsto}{.}{.}{.}{.}{>}
\def\Z{\mathbf{Z}}
\def\Q{\mathbf{Q}}
\def\F{\mathbf{F}}
\def\C{\mathbf{C}}

\def\EA{E\kern-0.1em{A}}
\def\EB{E\kern-0.1em{B}}
\def\EC{E\kern-0.1em{C}}

\def\GL{\mathrm{GL}}
\def\BG{\mathrm{B}\Gamma}
\def\BGL{\mathrm{BGL}}
\def\BSL{\mathrm{BSL}}

\def\SL{\mathrm{SL}}

\def\PSL{\mathrm{PSL}}

\def\SU{\mathrm{SU}}

\def\Pic{\mathrm{Pic}}

\def\Htw{\widetilde{H}}
\def\Ktw{\widetilde{K}}
\def\SKtw{\widetilde{SK}}

\newtheorem{theorem}{Theorem}[section]

\newtheorem{df}[theorem]{Definition}

\newtheorem{conj}[theorem]{Conjecture}

\newtheorem{lemma}[theorem]{Lemma}
\newtheorem{sublemma}[theorem]{Sublemma}
\newtheorem{prop}[theorem]{Proposition}
\newtheorem{corr}[theorem]{Corollary}
\newtheorem{example}[theorem]{Example}
\newtheorem{remark}[theorem]{Remark}

\begin{document}

\title{The stable homology of congruence subgroups} 
\author{Frank Calegari}
\thanks{The author was supported in part by NSF Career Grant DMS-0846285
and NSF Grant DMS-1404620.}
\maketitle

\subsection{Introduction}

Let $F$ be a number field, and let $\Gamma_N = \SL_N(\OL_F)$. For an integer $M$, let $\Gamma_N(M)$ denote the principal congruence subgroup of level $M$, that is, the kernel of the mod-$M$
reduction map~$\Gamma_N \rightarrow \SL_N(\OL_F/M)$. The cohomology and homology groups of $\Gamma_N$ in any fixed degree are well known to be stable as $N \rightarrow \infty$~\cite{CharneyStable}. However, the cohomology  and homology groups
of $\Gamma_N(M)$  do
not stabilize as  $N \rightarrow \infty$ for trivial reasons (stability already fails for $H_1(\Gamma_N(p),\F_p)$).
Given the important role that the cohomology of congruence subgroups plays in the
Langlands programme~\cite{Ash,Scholze,CG}, it is of interest to see 
whether the failure of stability can
be repaired in some way.
Various alternatives have been suggested,
including the notion of \emph{representation stability} by Church and Farb~\cite{CF} (see also~\cite{CEF,CFN}, and~\cite{Putman}). The starting point for representation stability of arithmetic groups is the observation that
the group
$H_1(\Gamma_N(p),\F_p)$ for $F = \Q$ may be identified with the adjoint representation of 
$\SL_N(\F_p)$  
for all $N > 2$, and that this description  is, in some sense, independent of $N$. For example, one may  ask whether $H_d(\Gamma_N(p),\F_p)$ admits a similar such description 
for all sufficiently large~$N$ (\cite{CF}, Conjecture~8.2.1).
In contrast, the approach of~\cite{CEA}
is to instead consider the  \emph{completed} homology and cohomology groups~\cite{CEB}:
$$\Htw_*(\SL_N,\F_p):= \lim_{\leftarrow} H_*(\Gamma_N(p^r),\F_p)  \qquad \Htw^*(\SL_N,\F_p) :=\lim_{\rightarrow} H^*(\Gamma_N(p^r),\F_p),$$
$$\Htw_*(\SL_N,\Z_p):= \lim_{\leftarrow} H_*(\Gamma_N(p^r),\Z_p)  \qquad \Htw^*(\SL_N,\Q_p/\Z_p) :=\lim_{\rightarrow} H^*(\Gamma_N(p^r),\Q_p/\Z_p).$$
There is a natural duality isomorphism $\Htw^*(\SL_N,\F_p) = \Hom(\Htw_*(\SL_N,\F_p),\F_p)$ from the universal coefficient theorem, and a
 corresponding isomorphism 
 $$\Htw^*(\SL_N,\Q_p/\Z_p) =  \Hom(\Htw_*(\SL_N,\Z_p),\Q_p/\Z_p).$$
One has the following  result from~\cite{CEA} (the generalization to number fields is immediate):

\begin{theorem} \label{theorem:MF} 
The modules $\Htw_{d}(\SL_N,\Z_p)$ stabilize as $N \rightarrow \infty$; the corresponding
limits:
$$\Htw_{d}(\SL,\Z_p):= \lim_{N \rightarrow \infty} \Htw_{d}(\SL_N,\Z_p)$$
are  finitely generated $\Z_p$-modules. Moreover, the action of 
$\SL_N(\OL_p):=\displaystyle{\prod_{v|p} \SL_N(\OL_v)}$ on $\Htw_{d}(\SL_N,\Z_p)$
is trivial for sufficiently large~$N$.
\end{theorem}

The perspective of this paper is that  the groups $\Htw^*$  and $\Htw_*$ are not merely a convenient way to package information concerning the classical cohomology and homology groups, but are instead the correct object of study. 
From an arithmetic perspective,  the classical cohomology groups $H^*(\Gamma_N(p^r),\F_p)$ should carry
interesting information about the field $\OL_F$. 
Let $G_N(p^r)$ be the kernel of the map $\SL_N(\OL_p) \rightarrow \SL_N(\OL/p^r)$; 
there is an inflation map $H^*(G(p^r),\F_p) \rightarrow H^*(\Gamma_N(p^r),\F_p)$. The
ultimate reason for
the failure of stability of classical congruence subgroups
is that the source of this map is unstable. On the other hand, the source only
contains \emph{local} information concerning $F$, in particular,
 it only depends  on the  decomposition of $p$ in $\OL_F$\footnote{In this paper, we use \emph{local} and
 \emph{global} in the sense used by number theorists, namely, to distinguish quantities which only
 depend on the embeddings of $F$ into either $\Qbar_p$ or $\mathbf{C}$ from those that
 depend on more subtle invariants of $F$. For example, the ranks of the $K$-groups $K_n(\OL_F) \otimes \Q$ are local because they are determined
 by the number of real and complex embeddings of $F$~\cite{Borel}. In contrast, the torsion subgroups of
 $K_n(\OL_F)$ are global.}.
 Hence, by taking the \emph{direct limit over $r$}
(or inverse limit in the case of homology), one  excises all the local terms and arrives at a group that both
contains all the interesting global information and is stable in $N$. From this optic, the relationship
between 
representation stability (in this particular context) and arithmetic disappears and, in particular,
representation stability seems more to be a phenomenon related to the cohomology of $p$-adic Lie groups rather than
 arithmetic groups.
 Theorem~\ref{theorem:MF} then says, in effect, that once one removes the uninteresting local terms from
 the cohomology of congruence subgroups, what remains is stable and finitely generated over $\Z_p$, and carries, 
 as we shall demonstrate,  interesting arithmetic information.

\medskip

An instructive and elementary example of what information is contained in completed cohomology can
already be seen in degree one, providing we consider the group $\GL_N$ rather than $\SL_N$. 
Suppose that $N$ is at least three. For convenience,
let us also suppose that  $F$
admits a real embedding, so that the congruence subgroup property holds
(exactly) 
for $\SL_N(\OL_F)$ (see~\cite{Serre}, Theorem~3.6); more generally, by the same
reference, if we assume that $p$ is prime to the order~$w_F$ of the finite group 
 of roots
of unity in $F$, then the congruence kernel is abelian of order prime to~$p$.
The group~$H_1(\SL_N(\OL_F),\Z)$ may be identified with the abelianization of~$\SL_N(\OL_F)$.
The congruence subgroup property implies (under our assumptions) that that any finite  quotient
of order prime to~$w_F$
must factor through the map~$\SL_N(\OL_F) \rightarrow \prod_{S} \SL_N(\OL_{F,v})$
for some finite set of places~$v \in S$. 
However, the abelianization of~$\SL_N(\OL_{F,v})$ is trivial for any finite place~$v$ (assuming that~$N > 2$), and so~$H_1(\SL_N(\OL_F),\Z_p)$ is trivial.
The same argument shows that for the congruence
subgroup~$\Gamma_N(p^r) \subset \SL_N(\OL_F)$, there is an isomorphism
$$H_1(\Gamma_N(p^r),\Z_p) = \Gamma_N(p^r)/\Gamma_N(p^{2r}) \simeq G_N(p^r)/G_N(p^{2r}).$$
It follows that the natural maps
$H_1(\Gamma_N(p^{2r}),\Z_p) \rightarrow H_1(\Gamma_N(p^r),\Z_p)$
will be zero, and so, in particular, the completed 
homology group~$\Htw_1(\SL_N,\Z_p)$ vanishes.
This implies that the completed homology group~$\Htw_1(\GL_N,\Z_p)$
may be identified under the determinant map with the completed
first homology groups of the unit group~$\OL^{\times}_F = \GL_1(\OL_F)$.
Let~$\GL_N(\OL_F,p^r) \subset \GL_N(\OL_F)$ denote the principal
congruence subgroup of level~$p^r$. 
Then there is an equality
$$ \Htw_1(\GL_N,\Z_p) := \lim_{\leftarrow} H_1(\GL_N(\OL_F,p^r),\Z_p) = \ker \left( \OL^{\times}_F \otimes \Z_p \rightarrow
\prod_{v|p} \OL^{\times}_v\right).$$
The claim that the right hand side vanishes is exactly the statement of Leopoldt's
conjecture, which is  well known to be
a deep open problem concerning the arithmetic of the field $F$. 
More precisely, our assumptions on roots of unity in~$F$ imply that~$\OL^{\times}_F \otimes \Z_p$
is torsion free; the usual statement of Leopoldt's conjecture is that the map:
$$\OL^{\times}_F \otimes \Q_p \rightarrow \prod_{v|p} F^{\times}_v$$
is injective.
For those unfamiliar with Leopoldt's conjecture, it may be helpful to imagine
replacing the prime~$p$ by the prime~$\infty$. The corresponding statement is then
the injectivity of the map
$$\OL^{\times}_F \otimes \R \rightarrow \prod_{v| \infty} F^{\times}_v.$$
This map is exactly the classical regulator map ($u \mapsto (\log |u|_v)$) whose injectivity
was established by Dirichlet as the first step in proving the unit theorem.
In this paper, we shall argue how analogous questions and conjectures
 concerning the groups~$\Htw_d(\SL,\Z_p)$
for higher~$d$  concern the (conjectural) injectivity of higher~$p$-adic regulator maps,
and bear the same relation to Borel's higher regulator maps as Leopolodt's conjecture
bears to the classical regulator map. (Why the~$p$-adic versions of these maps
seem much harder to control than their real analogues  is somewhat
of a mystery.)

\medskip

Theorem~\ref{theorem:MF} gives a very good qualitative description of completed cohomology
in  the stable range. 
The main concern of this paper, which can be considered a sequel to~\cite{CEA}, is then to address the
groups $\Htw_*$ in a \emph{quantitative} manner.
One may  also ask whether information concerning $\Htw_*$ or $\Htw^*$ can be translated into information 
concerning the classical cohomology groups.
For example, Benson Farb asked the author whether Theorem~\ref{theorem:MF} can be utilized for
the \emph{computation} of explicit cohomology groups.
In order to answer the broader question of what it might mean to \emph{compute} classical
cohomology groups,
we first recall what happens at full level. The stable homology of $\GL_N(\Z)$ may be identified with the homology of
$\BGL(\Z)^{+}$, and the \emph{homotopy} groups of $\BGL(\Z)^{+}$ for $n \ge 1$ are the algebraic $K$-groups
$K_n(\Z)$, which are finitely generated over $\Z$ and are
 completely known, at least in terms of Galois cohomology groups (see Theorem~\ref{theorem:V}). 
 The ranks of the \emph{rational} cohomology groups are
determined by the signature of the field
$F$ by a theorem of Borel~\cite{Borel}.
In this paper, we show that:
\begin{itemize}
\item The groups $\Htw_*$  are the continuous homology groups of a certain homotopy fibre
 $\Yoon$ (Theorem~\ref{theorem:content}) whose homotopy groups with coefficients
 in $\Z_p$ we can 
 calculate rationally (Lemma~\ref{lemma:qm}), and even integrally in many cases, in terms
 of Galois cohomology groups.
 \item 
Assuming, in addition, a generalization of Leopoldt's conjecture, we can also give
(Theorem~\ref{theorem:padicborel}) a complete description
of $\Htw_*(\SL,\Q_p):=\Htw_*(\SL,\Z_p) \otimes \Q$; note that this is different from the inverse  limit with coefficients in $\Q_p$, which,
by Borel's theorem, coincides with the stable rational homology at level one.	
\item For $F = \Q$, we compute $H_2(\Gamma_N(p),\F_p)$ for sufficiently large $N$ 
(Corollary~\ref{corr:benson}), answering a question of Farb).
\item For very regular primes (Definition~\ref{df:veryregular}) $p = \p \pbar$ in an imaginary quadratic field, we compute $H_*(\Gamma_N(\p^m),\Z_p)$ in the stable range explicitly (Theorem~\ref{theorem:exp}).
\end{itemize}
We give an example of a  precise theorem  we can state now. 
Let $G_S$ denote the Galois group of the maximal Galois extension
of $F$ unramified away from primes dividing $p$ and $\infty$.
Let~$w_F$ denote the  number of roots of unity in~$F$.

\begin{theorem}  \label{theorem:two}  If $p$ 
 does not divide $w_F \cdot |K_2(\OL_F)|$, then there
is an isomorphism
$$\Htw^2(\SL,\Q_p/\Z_p)  \simeq H^1(G_S,\Q_p/\Z_p(-1)).$$
For all $p$, the equality holds up to a finite group.
If $F = \Q$ and $p \ge 3$, then
$\Htw^2(\SL,\Q_p/\Z_p) = \Q_p/\Z_p$ and
$\Htw_2(\SL,\Z_p) \simeq \Z_p$. 
\end{theorem}

We expect that the equality holds without any assumption on the order
of $K_2(\OL_F)$, and this would follow from Conjecture~\ref{conj:answer}. (Note, however,
that~$K_2(\OL_F)$ is always finite.)
This theorem was inspired by the results of~\cite{CV}, particularly Chapter~8, which suggests a link between
classes in the cohomology of $H^2(\Gamma_N(M),\Q_p/\Z_p)$ of congruence subgroups and classes in
$H^1(F,\Q_p/\Z_p(-1))$  unramified outside $M$ (see, in particular, the discussion in~\S8.3 of \emph{ibid}).

\subsection{Acknowledgements}  I would like to thank Benson Farb for asking the question that led to this 
paper.
I would also like to thank Matthew Emerton, Eric Friedlander, Paul Goerss, Andrei Suslin,  Akshay Venkatesh, and Chuck Weibel  for useful conversations, in particular to Eric Friedlander to suggesting the use of the Zeeman comparison theorem,  
and 
Paul Goerss for  suggesting that the Eilenberg--Moore spectral sequence was
(in certain contexts) a better way to compute the cohomology of a fibre in a fibration than the Serre spectral sequence. 
I would also like to thank Toby Gee (and other members of the Imperial reading group) for some helpful comments on an earlier version of this manuscript.
The influence of~\cite{CV} on this paper should also be clear.
Finally, I would like to thank the tremendous job of the referee, who, apart from a very
thorough report which helped
to correct a number of minor errors, lacunae, and other inaccuracies (both mathematical and pedagogical),  also 
taught me
some algebraic topology.

\section{Homotopy}

\subsection{Completed \texorpdfstring{$K$}{K}-theory}
		In this section, we define the completed $K$-groups $\Ktw_*(\OL)$ of $\OL$ with respect to a prime $p$, at least for the ring of integers~$\OL = \OL_F$ of a number field~$F$.
Here the notion of ``completion'' refers (in principle) to replacing $\Gamma_N = \GL_N(\OL)$ by
the congruence subgroups $\Gamma_N(p^m)$ for arbitrarily large $m$, rather than other possible forms
of completion --- the notation is chosen via the analogy with the completed homology groups
$\Htw_*$. 
The groups~$\Ktw_n(\OL)$ are essentially, depending on the parity of $n$, either
the kernel or cokernel  of Soul\'{e}'s $p$-adic regulator map~\cite{SouleReg}.

For any prime $p$ and integer $m$,  define~$X_m = X_m(\OL,p^m)$
to be the  homotopy fibre of the following map:
$$\begin{diagram}
{\BGL(\OL)}^{+} 
& \rTo & {\BGL(\OL/p^m)}^{+} \end{diagram}$$
The homotopy groups of $X_m$ live in a long exact sequence with the $K$-groups
of $\OL$ and of $\OL/p^m$ (at least for~$n \ge 1$). In particular, by a result of Quillen~\cite{QuillenFinite},
they are finitely generated (see also Corollary~3.8 of~\cite{CharneyStable}).
The map $B \Gamma(p^m) \rightarrow X_m$ is \emph{not}, however,  a
homological equivalence 
--- this is
the so called ``failure of excision.''
(This remark was made in the introduction to~\cite{Charney}.) The idea behind this paper is that,
since the \emph{limiting} groups $\Htw_*$ have a trivial
 action of~$\SL(\OL_p)$,  then some analogue of excision
will hold\footnote{Literally, of course, this is not true.
 Since the homotopy fibre of~$B \GL(\OL) \rightarrow B \GL(\OL/p^m)$ 
is~$B \Gamma(p^m)$, one way of rephrasing the difficulty above is that the~$+$
construction does not commute with taking homotopy fibres. The 
homotopy fibre of~$B \GL(\OL) \rightarrow B \GL(\OL_p)$, however, is essentially
the discrete set~$\GL(\OL_p)/\GL(\OL)$ --- the map on~$\pi_1$ is injective and there
are no higher homotopy groups. The plus construction applied to the discrete
set~$\GL(\OL_p)/\GL(\OL)$ (whatever the plus construction would mean here), however, is certainly
not the homotopy fibre of~$\BGL(\OL)^{+} \rightarrow \BGL(\OL_p)^{+}$.}
for $m = \infty$.

\medskip

Since~$\pi_0(\BGL(R)^{+}) = 0$  does not correspond to~$K_0(R)$, it
is technically more convenient to work with the homotopy fibrations of~$K$-theory
spaces:
$$
\begin{diagram}
K(\OL,p^m)  & \rTo & K(\OL) & \rTo & K(\OL_p/p^m), \\ 
K(\OL,\OL_p)  & \rTo &K(\OL) & \rTo & K(\OL_p).
\end{diagram}
$$
Note that~$K(R,I)$ usually denotes the homotopy fibre of the map~$K(R) \rightarrow K(R/I)$,
and this is the meaning of~$K(\OL,p^m)$;  by
abuse of notation, we let~$K(\OL,\OL_p)$ be the homotopy fibre of~$K(\OL)
\rightarrow K(\OL_p)$; hopefully
no confusion will arise.

\medskip

We make the following definition:

\begin{df} The completed $K$-groups $\Ktw_n(\OL):=K_n(\OL,\OL_p;\Z_p)$ with respect to a prime $p$ are, for~$n \ge 1$, 
the homotopy groups~$\pi_n(K(\OL,\OL_p);\Z_p)$. 
\end{df}

We make below (Definition~\ref{df:zero}) an ad hoc definition of~$\Ktw_0(\OL,\OL_p;\Z_p)$. 

\begin{lemma} \label{lemma:longlong}
Let~$\OL = \OL_F$ for a number field~$F$. There is a long exact sequence:
$$\rightarrow K_n(\OL,\OL_p;\Z_p) \rightarrow
K_n(\OL) \otimes \Z_p \rightarrow K_n(\OL;\Z_p) \rightarrow
K_{n-1}(\OL,\OL_p;\Z_p) \rightarrow
\ldots $$
\end{lemma}

\begin{proof} 
We first note that all the relevant spaces~$K(\OL,\OL_p)$, $K(\OL)$, $K(\OL_p)$, etc. we
 consider are infinite loop spaces.
We now  recall some basic constructions,
definitions, and theorems concerning~$K$-theory with coefficients.
\begin{enumerate}
\item (\cite{WB}, IV.2.1):  The homotopy groups with coefficients in $\Z/p^m$ are defined to be
$$\pi_n(X;\Z/p^m):= [P^n(\Z/p^m),X],$$ where $P^n(\Z/p^m)$ 
is the space formed from the sphere~$S^{n-1}$ by attaching an $n$-cell via a degree~$p^m$ map.
In general, this space is a group for~$n \ge 3$ and a set for~$n = 2$, however,
it is well defined as a group for all~$n \ge 1$ when~$X$ is an infinite loop space.
Given a Serre fibration, there is a naturally associated long exact sequence.
\item (\cite{WB}, IV.2.9): The $p$-adically completed homotopy groups $\pi_n(\EE;\Z_p)$ of
a spectrum $\EE$ are defined
to be the homotopy groups of the homotopy limit $\widehat{\EE}$ of 
$\EE \wedge P^{\infty}(\Z/p^m)$. If the $\pi_*(\EE;\Z/p^m)$ are finite for all~$m$, then
$$\pi_*(\EE;\Z_p) \simeq \projlim \pi_*(\EE;\Z/p^m).$$
\item   (\cite{WB}, IV.2.2): There is a universal coefficient sequence:
$$0 \rightarrow \pi_n(X) \otimes \Z/p^m \rightarrow \pi_n(X;\Z/p^m) \rightarrow
\pi_{n-1}(X)[p^m] \rightarrow 0.$$
\end{enumerate}

Note that $K_n(\OL):=\pi_n(K(\OL))$ for all~$n$.
 The groups $K_n(\OL)$ for~$\OL = \OL_F$ the ring
of integers of a number field are known to be finitely generated abelian groups, and hence the
projective limit 
$\projlim K_{n-1}(\OL)[p^m]$ vanishes. It follows from the universal
coefficient sequence that
$$K_n(\OL;\Z_p) = \projlim K_n(\OL;\Z/p^m) = K_n(\OL) \otimes \Z_p.$$
The groups~$K_n(\OL_p;\Z/p^r)$ are similarly finite (by Theorem~$A$ of~\cite{local}),
and so
$$K_n(\OL_p;\Z_p) = \projlim K_n(\OL_p;\Z/p^m).$$
The Serre long exact sequence gives a sequence
$$\ldots \rightarrow \pi_n(K(\OL,\OL_p);\Z/p^m))
\rightarrow K_n(\OL;\Z/p^m)  \rightarrow  K_n(\OL_p;\Z/p^m) \rightarrow \ldots $$
By comparison with the surrounding terms, 
we see that the~$\Z_p$-modules~$\pi_n(K(\OL,\OL_p),\Z/p^m)$ are finite
and also have bounded rank (as~$\Z_p$-modules) for all~$m$.
Hence, taking an inverse limit in~$m$ and replacing~$K_n(\OL;\Z_p)$
by~$K_n(\OL) \otimes \Z_p$, we obtain the desired exact sequence.
(The finiteness properties guarantee that the relevant Mittag--Leffler conditions are automatically
satisfied, and so there are no issues concerning~${\lim}^1$.)
For the tail of the long exact sequence involving terms with~$n = 0$, see the remarks
after Definition~\ref{df:zero}.
\end{proof}

It will also be useful to work with the homotopy fibrations
\begin{equation} \label{eq:fibration}
\begin{diagram}
\Ymn = SK(\OL,p^m) & \rTo & SK(\OL) & \rTo &  SK(\OL/p^m). \\ 
\Yoon = SK(\OL,\OL_p) & \rTo & SK(\OL) & \rTo & SK(\OL_p). 
\end{diagram}
\end{equation}
Since $\pi_n(\BSL(R)^{+}) \simeq \pi_n(\BGL(R)^{+}) = \pi_n(SK(R))$ for $n \ge 2$, one obtains isomorphisms 
$\pi_n(\Yoon;\Z_p) \simeq K_n(\OL,\OL_p;\Z_p)$ for $n \ge 2$ and $m \in \mathbf{N} \cup \infty$.
Working with $\Yoon$ allows us to apply the Hurewicz Theorem more usefully, since, for many
fields $F$, the fibre $\Yoon$ will be simply connected.

\subsection{Tate Global Duality}

Let $S = \{v|p \} \cup \{v|\infty\}$, and let $G_S$ denote the Galois group of the maximal
extension of $F$ outside $S$. If $M$ is a finite $G_S$-module of odd order,  then
$H^*(G_S,M) \simeq H^*(\Spec(\OL[1/p]),M)$ is finite.
If $n$ is a positive integer, let $\Z_p(n)$ denote the $n$th Tate twist of $\Z_p$, that is, the module $\Z_p$ such that the action of
 $G_{\Q}$ is via the $n$th power of the cyclotomic character. Similarly, the module $\F_p(n)$ denotes the
corresponding twist of the module $\F_p$.
  For a compact
 or discrete $\Z_p$-module $A$, let $A^{\vee} = \Hom_{\Z_p}(A,\Q_p/\Z_p)$ denote the
 Pontryagin dual of $A$ (so canonically $A^{\vee \vee} \simeq A$). For a compact or discrete $G_S$-module $M$,
 let $M^*$ denote the twisted Pontryagin dual $\Hom_{\Z_p}(M,\Q_p/\Z_p(1))$.
 For a finite $G_S$-module $M$, recall
 (\cite{Milne,TateDual}) that by Poitou--Tate duality there is an exact sequence as follows:
\begin{equation} \label{eq:PT}
\begin{diagram}
0 & \rTo & H^0(G_S,M) & \rTo & \prod_{S} H^0(G_{v},M) & \rTo & H^2(G_S,M^*)^{\vee}   & & \\
 & & & & & & \dTo  & &  \\
 & & H^1(G_S,M^*)^{\vee} & \lTo & \prod_{S} H^1(G_v,M) & \lTo & H^1(G_S,M)   & & \\
 & &  \dTo  & & & & & & \\
 & & H^2(G_S,M) & \rTo & \prod_S H^2(G_v,M) & \rTo & H^0(G_S,M^*)^{\vee} & \rTo & 0 \\
 \end{diagram}
 \end{equation}
By taking limits, one also obtains a corresponding sequence for compact $M$, where cohomology
of compact or discrete modules is taken in the usual continuous sense. 
Specifically, let $n$ be a positive integer and let $M = \Z_p(n)$. 
Then $M^* = \Hom(M,\Q_p/\Z_p(1)) = \Q_p/\Z_p(1-n)$. Moreover, $H^0(G_S,M) = H^0(G_v,M) = 0$.
We thus obtain the following exact sequence:

$$
\begin{diagram}
0 & \rTo & H^2(G_S,\Q_p/\Z_p(1-n))^{\vee} & \rTo & H^1(G_S,\Z_p(n)) &
\rTo & \prod_{S} H^1(G_v,\Z_p(n))   \\
& \rTo & H^1(G_S,\Q_p/\Z_p(1-n))^{\vee} & \rTo & H^2(G_S,\Z_p(n)) & \rTo & \prod_S H^2(G_v,\Z_p(n))  
 \\
& \rTo & H^0(G_S,\Q_p/\Z_p(1-n))^{\vee}   &\rTo   & 0. & & \\
\end{diagram}
$$
The following result follows from work of many, including
Soul\'{e}~\cite{Soule2},  Madsen and Hesselholt~\cite{local},
Voevodsky, Rost, Suslin,  Weibel and others~\cite{KZ,Voe}:
\begin{theorem} \label{theorem:V} Let $p > 2$. Then one has isomorphisms for $n > 1$ as follows:
$$K_{2n-1}(\OL) \otimes \Z_p \simeq H^1(G_S,\Z_p(n)), \qquad K_{2n-2}(\OL) \otimes \Z_p  \simeq 
H^2(G_S,\Z_p(n)),$$
$$K_{2n-1}(\OL_p; \Z_p) \simeq \prod_{S} H^1(G_v,\Z_p(n)), \qquad K_{2n-2}(\OL_p; \Z_p)  \simeq 
\prod_{S} H^2(G_v,\Z_p(n)).$$
Moreover, the maps $K_n(\OL) \otimes \Z_p \rightarrow K_n(\OL_p;\Z_p)$ induce the natural maps on cohomology.
\end{theorem}
For specific references, we refer to Theorem~A of~\cite{local} for the~$K$-theory
of (rings of integers in) local fields (see also Theorem~61 of~\cite{KZ}), and
Theorem~70 of  Weibel's survey~\cite{KZ} for the~$K$-theory of global fields. We note
that, in the latter reference, the Galois cohomology groups have been excised  from
the statement for the benefit of topologists, but they are easily extracted from the argument.
The compatibility of the these isomorphisms follows from the compatibility of the
corresponding motivic spectral sequences.
Note that for $p = 2$, the equalities hold up to a (known) finite group, and hold
on the nose if $F$ is totally imaginary.  For the groups~$K_1$, which are known classically,
one must make a minor adjustment to these descriptions in terms of Galois cohomology:
 this amounts (for experts) to replacing the Galois cohomology groups~$H^1$ with
the Bloch--Kato groups~$H^1_{\fff}$. More prosaically, there are  Kummer isomorphisms
$$H^1(G_S,\Z_p(1)) = (\OL_F[1/p])^{\times} \otimes \Z_p,
\qquad H^1(G_v,\Z_p(1)) = F^{\times}_v \otimes \Z_p,$$
whereas the~$K_1$ groups we are interested in
should be identified with the groups
$$H^1_{\fff}(G_S,\Z_p(1)) = (\OL_F)^{\times} \otimes \Z_p,
\qquad H^1_{\fff}(G_v,\Z_p(1)) = \OL^{\times}_v \otimes \Z_p$$
respectively. This discrepancy is related to the fact that~$K_*(\OL_F[1/p])$ and~$K_*(\OL_F)$
coincide in higher degrees but not in degree one.
We now have:
\begin{lemma} \label{lemma:qm}  Define the
groups $\Ktw^{?}_n(\OL)$ as follows:
$$\Ktw^{?}_{2n-1}(\OL):= H^2(G_S,\Q_p/\Z_p(1-n))^{\vee} \oplus  H^0(G_S,\Q_p/\Z_p(-n))^{\vee},$$
$$\Ktw^{?}_{2n-2}(\OL):=  H^1(G_S,\Q_p/\Z_p(1-n))^{\vee}.$$
Then there is an exact sequence as follows:
$$\dotsc \rightarrow \Ktw^{?}_{n}(\OL) \rightarrow  K_n(\OL) \otimes \Z_p  \rightarrow
K_n(\OL_p; \Z_p) \rightarrow \Ktw^{?}_{n-1}(\OL) \rightarrow \dotsc $$
There are rational isomorphisms
$\Ktw_n(\OL) \otimes \Q = \Ktw^{?}_n(\OL) \otimes \Q$,
 If either $K_n(\OL)$ is finite of order prime to $p$ or $K_{n+1}(\OL_p;\Z_p) = 0$, 
 then there is an isomorphism $\Ktw_n(\OL) =  \Ktw^{?}_n(\OL).$
\end{lemma}

\begin{proof} The exact sequence follows directly from the Poitou--Tate sequence~(\ref{eq:PT}) above
(one also has to make an easy check for~$n = 1$).
 It follows that \emph{if} there existed   maps
$\Ktw_n(\OL) \rightarrow \Ktw^{?}_n(\OL)$ inducing a commutative map between long exact sequences,
 then
 $\Ktw_n(\OL) \simeq \Ktw^{?}_n(\OL)$ by the $5$-lemma. Such a map exists after tensoring with
$\Q$ since $\Q_p$ is projective. If $K_n(\OL) \otimes \Z_p = 0$, then both $\Ktw_{n}(\OL)$
and $\Ktw^{?}_n(\OL)$ are isomorphic  to the quotient of $K_{n+1}(\OL_p;\Z_p)$ by the image of
$K_{n+1}(\OL) \otimes \Z_p$. Similarly, if $K_{n+1}(\OL_p;\Z_p) = 0$, then both
$\Ktw_{n}(\OL)$
and $\Ktw^{?}_n(\OL)$ are isomorphic  to the kernel of the map
$K_n(\OL) \otimes \Z_p \rightarrow K_n(\OL_p;\Z_p)$.
\end{proof}

We make the following conjecture:

\begin{conj}  \label{conj:answer} There is an isomorphism $\Ktw_n(\OL) \simeq \Ktw^{?}_n(\OL)$.
\end{conj}

This seems natural enough for~$n$ even.  For~$n$ odd, the conjecture
also seems natural in light of the fact that we expect the term~$H^2(G_S,\Q_p/\Z_p(1-n))^{\vee}$
to vanish (see also Remark~\ref{remark:van}).

\begin{remark} \emph{There is some hope to prove this statement by constructing a natural map
$\Ktw_n(\OL) \rightarrow \Ktw^{?}_n(\OL)$.  The author has had some discussions with Matthew Emerton
regarding this question, and we hope to return to it in the future.}
\end{remark}

\begin{df}[The group~$\Ktw_0(\OL)$] \label{df:zero}
\emph{We now extend our definitions to~$n = 0$. For~$n = 0$, let
$$\Ktw_0(\OL):=\Ktw^{?}_0(\OL) = H^1(G_S,\Q_p/\Z_p)^{\vee}.$$
}
\end{df}

There is an isomorphism~$H^1(G_S,\Q_p/\Z_p)^{\vee}
= \Hom(G_S,\Q_p/\Z_p)^{\vee} = (G_S)^{\mathrm{ab}} \otimes \Z_p$.
There are isomorphisms
$K_1(\OL) \simeq \OL^{\times}$ and~$K_1(\OL_p;\Z_p) =  \Z_p \otimes \prod_{v|p} \OL^{\times}_p$.
Moreover,~$K_0(\OL) \simeq \Pic(\OL) =  \Cl(\OL) \oplus \Z$.
From class field theory, there is an exact sequence:
$$ \OL^{\times} \otimes \Z_p \rightarrow  \left( \prod_{v|p} \OL^{\times}_p \right) \otimes \Z_p
\rightarrow  (G_S)^{\mathrm{ab}} \otimes \Z_p \rightarrow \Cl(\OL) \otimes  \Z_p \rightarrow 0.$$
Hence the long exact sequence of
Lemma~\ref{lemma:longlong} continues as far as:
$$\ldots \rightarrow K_1(\OL) \otimes \Z_p
\rightarrow K_1(\OL_p;\Z_p) \rightarrow \Ktw_0(\OL)
\rightarrow  K_0(\OL) \otimes \Z_p \rightarrow \Z_p \rightarrow 0.$$

\medskip

After tensoring with $\Q$, the long exact sequence of Poitou--Tate and hence of $K$-groups breaks
up into  exact sequences of length $6$ (the~$H^0 \otimes \Q$ term vanishes for~$n > 1$).
One immediately obtains the following:
\begin{lemma} \label{lemma:comp} There is an equality:
$$\dim \Ktw_{2n-2}(\OL) \otimes \Q - \dim \Ktw_{2n-1}(\OL) \otimes \Q = \begin{cases} r_1 + r_2, & n > 1 \ \text{even}, \\ r_2, & n > 1 \ \text{odd}, \\
r_2 + 1, & n = 1.
\end{cases}$$
\end{lemma}

\begin{proof}
By computing the Euler characteristic of the six term exact sequence, the result
follows from Theorems~\ref{theorem:borel} and~\ref{theorem:wagoner} below.
Alternatively,
the result follows for $\Ktw^{?}_*$ by the global Euler characteristic formula~\cite{Milne,TateDual}.
\end{proof}

A complete evaluation of the rank of $\Ktw_*$ would follow (and is equivalent to) from the following conjecture, which is 
already implicit
in the work of Soul\'{e}:

\begin{conj} \label{conj:L} For all $n$,
$\dim \Ktw_{2n-1}(\OL) \otimes \Q = 0$.
\end{conj}

For $n = 1$, this is Leopoldt's conjecture.
For totally real fields, Conjecture~\ref{conj:L}   is equivalent to the non-vanishing of a certain $p$-adic zeta function. This equivalence follows
from Theorem~3 of~\cite{SouleReg} (see also Remark 3.4 p.399 of \emph{ibid}).  
Actually, Soul\'{e} assumes that~$F$ is abelian, but the general case follows from the 
proof of the Quillen--Lichtenbaum conjecture (giving a 
cohomological description of the global
and local~$K$ groups) together with the main conjecture for totally real fields proved by Wiles~\cite{WilesMain}.
In the general case, it is equivalent
to showing that the kernel of the $p$-adic regulator map on $K_{2n-1}(\OL_F) \otimes \Q_p$ is zero.
The first non-trivial
case of this conjecture 
 for $F = \Q$ is $n = 3$, where (as noted above) it ``reduces'' to the question
of the non-vanishing of the
Kubota--Leopoldt zeta function $\zeta_p(3)$. One might actually make the stronger conjecture that
this number is irrational (see~\cite{Zeta}). For more discussion of injectivity of localization
maps, see~\cite{periods}.

\medskip

\begin{remark} \label{remark:van}
\emph{
Since $G_{S}$ has cohomological dimension $2$ (with coefficients in $\Z_p$ for $p \ne 2$), the
group $H^2(G_S,\Q_p/\Z_p(1-n))$ is divisible and hence its dual is torsion free. In particular, if
 Conjecture~\ref{conj:L} is true, then this $H^2$ term vanishes, and
 $$\Ktw^{?}_{2n-1}(\OL) =   H^0(G_S,\Q_p/\Z_p(-n))^{\vee} \simeq H_0(G_S,\Z_p(n)) =
  \Z_p/(a^n -1) \Z_p,$$
 where $a$ is any topological generator of $\Z^{\times}_p$. In particular, 
 $\Ktw^{?}_{2n-1}(\OL) = 0$ unless $n \equiv 0 \mod p-1$.
 }
 \end{remark}
 
 More generally, we have the following estimates:

\begin{lemma} \label{lemma:q}  Suppose that $F$ is a number field of degree $d$ and signature
$(r_1,r_2)$.
If $p$ does not divide the order of the class group of $F(\zeta_p)$, then, for $n > 0$, we have equalities:
$$\dim \Ktw_{n}(\OL) \otimes \Q = 
\begin{cases} r_2, & n \equiv 0 \mod 4, \\
r_1 + r_2, & n \equiv 2 \mod 4, \\
0, & n \equiv 1 \mod 2. \end{cases}
$$
If $r_2 = 0$, so $F$ is a totally real field of degree $d$, then unconditionally
$$\dim \Ktw_{4n-1}(\OL) \otimes \Q = 0, \qquad
\dim \Ktw_{4n-2}(\OL) \otimes \Q = d.$$
 In particular, 
 $\Ktw_2(\OL) \otimes \Q = \Htw_2 \otimes \Q = \Q^d_p$. 
\end{lemma}

\begin{proof} We begin by recalling Borel's theorem (Prop~12.2 of~\cite{Borel}):

\begin{theorem}[Borel]  \label{theorem:borel} Suppose that $F$ is a number field of degree $d$ and signature
$(r_1,r_2)$. For $n > 0$, we have equalities:
$$\dim K_{n}(\OL) \otimes \Q = 
\begin{cases} 
r_1 + r_2 - 1, & n = 1, \\
r_1 + r_2, & n \equiv 1 \mod 4 \ \text{and}  \ n > 1 \\
r_2, & n \equiv 3 \mod 4, \\
0, & n \equiv 0 \mod 2. \end{cases}
$$
\end{theorem}

We also have the following result (Theorem~61 of~\cite{KZ}), which 
(as noted in \emph{ibid.}) is essentially due to  Wagoner and Milgram~\cite{Wagoner}  and Panin~\cite{SOWA}.
(Alternatively, this result follows directly from local Tate duality and the Euler
characteristic formula
 given the identification
of these groups with Galois cohomology in Theorem~\ref{theorem:V}.)

\begin{theorem} \label{theorem:wagoner} Let~$p$ be a rational prime, and let~$F$ be a number field of degree~$d$
and signature~$(r_1,r_2)$.
For~$n > 0$, we have equalities:
$$\dim K_{n}(\OL_p;\Z_p) \otimes \Q = 
\begin{cases} 
d = r_1 + 2 r_2, & n  \equiv 1 \mod 2 \\
0, & n \equiv 0 \mod 2. \end{cases}
$$
\end{theorem}

All the equalities now follow from a diagram chase, assuming
 that the maps
 $$H^1(G_S,\Z_p(n)) \rightarrow \prod_{v|p} H^1(G_S,\Z_p(n))$$
are injective after tensoring with $\Q$.  Specifically, the groups~$\Ktw_n(\OL) \otimes \Q$
will then identified with the cokernel of the map from~$K_{n+1}(\OL;\Z_p) \otimes  \Q$
to~$K_{n+1}(\OL_p;\Z_p) \otimes \Q$, and hence have dimension zero when~$n$ is odd and,
for~$n > 1$,
dimension
$$\begin{aligned} (r_1 + 2 r_2) - (r_1 + r_2), & \qquad n \equiv 0 \mod 4, \\
(r_1 + 2 r_2) - r_2, \quad &  \qquad n \equiv 2 \mod 4. \end{aligned}$$
The injectivity of the map of Galois cohomology groups is trivial when the group on the left is actually
zero after tensoring with $\Q$, which accounts for the unconditional cases when $F$ is totally
real and $r_2 = 0$.
Thus we may assume that $p$ does not divide the class number of $F(\zeta_p)$. But any element in the kernel
of the map on cohomology 
 will give rise to unramified non-trivial extension classes of $\Z_p$ by $\Z_p(n)$. In particular, the first layer will generate an unramified extension of
 $F(\F_p(n)) \subseteq F(\zeta_p)$ of degree $p$,
which would contradict the assumption on the class number.
\end{proof}

Note that Conjecture~\ref{conj:L} would imply that the equality of ranks holds for all primes,  not just regular ones.
We also note that these regulator maps in the optic of Galois cohomology were also studied by Schneider~\cite{Schtwo,Schone}.
In particular, by  Satz~$3$ of~\S6 of~\cite{Schtwo}, one has:

\begin{prop} \label{prop:Sch} For any fixed number field field~$F$ and prime~$p$, Conjecture~\ref{conj:L} is true for all but finitely many~$n$.
\end{prop}

\subsection{Comparison of Homologies}

\begin{lemma} \label{lemma:spectral} Let $A_j$ and $B_j$ be a sequence of $\F_p$-modules with
trivial $G$-action, and suppose that~$A_0 = B_0 = \F_p$.
Suppose that one has spectral sequences over~$\F_p$:
$$\EA^2_{i,j}:=H_i(G,A_j) \Rightarrow C_{i+j}, 
\qquad \EB^2_{i,j} = H_i(G,B_j) \Rightarrow C_{i+j}$$
together with compatible maps
$\EA^m_{i,j} \rightarrow \EB^m_{i,j}$
for all $m \ge 2$ inducing an automorphism of $C_{i+j}$.
Suppose that the maps
$$H_i(G,A_j) = \EA^2_{i,j} \rightarrow \EB^{2}_{i,j} = H_i(G,B_j)$$
are induced from the maps
$$A_j = \EA^2_{0,j} \rightarrow \EB^{2}_{0,j} = B_j.$$
Then there is an isomorphism~$A_j \simeq B_j$ for all~$j$.
\end{lemma}

\begin{proof}
Since~$A_q$ is a trivial~$G$-module, there is a canonical isomorphism
$$\EA^2_{p,q} = H_p(G,A_q) = H_p(G,\F_p) \otimes A_q = H_p(G,A_0) \otimes H_0(G,A_q)
= \EA^2_{0,q} \otimes \EA^2_{p,0}.$$
The assumed compatibility implies that there is a commutative diagram as follows:
$$
\begin{diagram}
 \EA^2_{0,q} \otimes \EA^2_{p,0} &  \rEquals &
\EA^{2}_{p,q} \\
\dTo & & \dTo \\
 \EB^2_{0,q} \otimes \EB^2_{p,0} &  \rEquals &
\EB^{2}_{p,q} \\
\end{diagram}
$$
The result then follows from the Zeeman comparison theorem (Theorem~3.26 of~\cite{Spectral}).
(Note that, since the coefficient ring is a field, one may disregard the~$\Tor_1$
terms.)
\end{proof}
\begin{df} For a space $X$, define the
continuous homology groups $\Hc_*(X,\Z_p)$ to be
$\Hc_*(X,\Z_p) := \projlim H_*(X,\Z/p^n)$.
\end{df}

\begin{theorem}
\label{theorem:content}  There are isomorphisms
$$\begin{aligned}
\Htw_*(\SL,\F_p) \simeq &  \ \displaystyle{\lim_{\leftarrow} H_*(\Ymn,\F_p) \simeq H_*(\Yoon,\F_p),} \\
\Htw_*(\SL,\Z/p^r \Z) \simeq &  \ \displaystyle{\lim_{\leftarrow} H_*(\Ymn,\Z/p^r \Z) \simeq H_*(\Yoon,\Z/p^r \Z),} \\
\Htw_*(\SL,\Z_p) \simeq  & \ \displaystyle{\lim_{\leftarrow} \Hc_*(\Ymn,\Z_p) \simeq  \Hc_*(\Yoon,\Z_p).}
\end{aligned}$$
\end{theorem}

\begin{proof}
Consider the following diagram:
$$
\begin{diagram}
\BG(p^m) & \rTo & \BSL( \OL) & \rTo & \BSL( \OL/p^m) \\
\dTo & & \dTo & & \dTo \\
Y_m & \rTo & \BSL( \OL)^{+} & \rTo & \BSL( \OL/p^m)^{+} \\
\uTo & & \dEquals & & \uTo \\
\Yoon & \rTo & {\BSL( \OL)}^{+} 
& \rTo &  {\BSL(\OL_p)}^{+} \\
\end{diagram}
$$
We obtain corresponding maps:
$$H_*(\Gamma(p^m),\F_p) \rightarrow H_*(Y_m,\F_p) \leftarrow H_*(\Yoon,\F_p).$$
 Moreover, we have a natural map of spectral sequences:

  \begin{equation} \label{eq:bb}
 \begin{diagram}
H_{i}(\SL(\OL/p^m),H_j(\Gamma(p^m),\F_p)) & \rSpectral & H_{i+j}(\Gamma,\F_p) \\
\dTo & & \dEquals \\
H_i(\SL(\OL/p^m),H_j(\Ymn,\F_p)) & \rSpectral & H_{i+j}(\Gamma,\F_p). \end{diagram}
 \end{equation}
 
\begin{sublemma}  \label{lemma:bbb} The~$E^{2}_{i, j}$ and hence~$E^{n}_{i,j}$ terms 
of the spectral sequence
$$
 \begin{diagram}
H_i(\SL(\OL/p^m),H_j(\Ymn,\F_p)) & \rSpectral & H_{i+j}(\Gamma,\F_p). \end{diagram}
$$
 have uniformly bounded dimension as~$m \rightarrow \infty$. 
 \end{sublemma}
 
We suspect that a  stronger claim holds, namely that the spectral sequence is constant
for sufficiently large~$m$, but  the boundedness is sufficient for our purposes.

 \begin{proof} 
We first assume that~$E^2_{i,0} = H_i(\SL(\OL/p^m),\F_p)$ is bounded for sufficiently large~$m$.
The claimed result then follows immediately for the first row. We now
  proceed by induction on the rows. The general row consists of
$\dim H_j(\Ymn,\F_p)$ copies of the first row, so it suffices to show that
this dimension is uniformly bounded. However, if~$E^2_{j,0} = H_j(\Ymn,\F_p)$ is
unbounded and
all the terms in lower rows are uniformly bounded, then~$E^{\infty}_{j,0}$
and hence~$H_j(\Gamma,\F_p)$ will
also be unbounded, which is  a contradiction. It thus remains to show that~$H_i(\SL(\OL/p^m),\F_p)$ is
bounded as~$m$ increases.
By classical stability~\cite{CharneyStable},
 we may replace~$\SL(\OL/p^m)$ by~$\SL_N(\OL/p^m)$ for some~$N$
depending only on~$i$. Let~$\SL_N(\OL/p^m,p) \subset  \SL_N(\OL/p^m)$ denote the kernel of the reduction map modulo~$p$. There is a spectral sequence:
 $$
 \begin{diagram}
H_i(\SL_N(\F_p),H_j(\SL_N(\OL/p^m,p),\F_p))  & \rSpectral & H_{i+j}(\SL_N(\OL/p^m),\F_p).
\end{diagram}
$$
It suffices to note that the coefficient system $H_j(\SL_N(\OL/p^m,p),\F_p)$ is independent of~$m$ for~$m \ge 1$
by~Corollary~2.34 of~\cite{Browder}.
\end{proof}

We can not deduce that $H_j(\Gamma(p^m),\F_p) \simeq H_j(\Ymn,\F_p)$ from 
equation~\ref{eq:bb}, exactly
because the action of $\SL(\OL/p^m)$ on $H_j(\Gamma(p^m),\F_p)$ is non-trivial.  The key point is thus
that, in the limit, the action of $\SL(\OL_p)$ on $\Htw_j$ \emph{is} trivial by 
Theorem~\ref{theorem:MF}.
 The diagram above gives rise to a compatible map of spectral sequences
$$
\begin{diagram}
\EA^2_{i,j} = H_i(\SL(\OL_p),\Htw_j(\SL,\F_p)) &  \rSpectral  & H_{i+j}(\Gamma,\F_p)  \\
\dTo & & \dEquals \\
\EB^2_{i,j} = H_i(\SL(\OL_p),\displaystyle{\lim_{\leftarrow} H_j(Y_m,\F_p)}) & \rSpectral & H_{i+j}(\Gamma,\F_p)  \\
\uTo & & \dEquals \\
\EC^2_{i,j} = H_i(\SL(\OL_p),\displaystyle{H_j(\Yoon,\F_p)}) & \rSpectral & H_{i+j}(\Gamma,\F_p) 
\
\end{diagram}
$$
The inverse limit on the second term commutes with the construction of the spectral sequence
because all the terms involved are uniformly bounded vector spaces over~$\F_p$ by
Sublemma~\ref{lemma:bbb}, and so all inverse limits  satisfy the Mittag--Leffler condition.
The action of $\SL(\OL_p)$ on $H_*(Y_m,\F_p)$ and
$H_j(\Yoon,\F_p)$ is trivial by
construction, and the action on $\Htw_*$ is trivial by Theorem~\ref{theorem:MF}. Hence, by Lemma~\ref{lemma:spectral},  one obtains isomorphisms
$\Htw_j(\SL,\F_p) \simeq H_j(\Yoon,\F_p)$.
By d\'{e}vissage, we obtain
isomorphisms 
$$\Htw_j(\SL,\Z/p^r \Z) \simeq \lim_{\leftarrow} H_j(\Ymn,\Z/p^r \Z) \simeq H_j(\Yoon,\Z/p^r \Z)$$
 for all $r$. Namely, we apply induction and compare the long exact sequences of homology associated
 to the short exact sequence
 $$0 \rightarrow \Z/p^{r-1} \Z \rightarrow \Z/p^r \Z \rightarrow \Z/p \Z \rightarrow 0,$$
and then apply the~$5$-lemma.
The groups $\Htw_j(\SL,\Z_p)$ are finitely generated over $\Z_p$ (by Theorem~\ref{theorem:MF}), and
hence coincide with the inverse limit of $\Htw_j(\SL,\Z/p^r \Z)$. 
\end{proof}

Note that the homology groups $H_*(\Yoon,\Z)$ are presumably quite badly behaved, thus 
$H_*(\Yoon,\Z) \otimes \Z_p$ presumably differs from $\Hc_*(\Yoon,\Z_p)$ in general.

\begin{remark} \emph{Combining Conjecture~\ref{conj:answer} with Lemmas~\ref{lemma:qm},~\ref{lemma:comp}
and Theorem~\ref{theorem:content}, we see that we have constructed an infinite loop space whose
homotopy groups with coefficients in~$\Z_p$ are given by~$\Ktw_*(\OL)$ and whose (continuous) homology groups are given by~$\Htw_*(\SL,\Z_p)$.
This should be thought of as completely analogous to the classical story, where the infinite loop space~$K(\OL)$ has homotopy groups~$K_n(\OL)$
and homology groups~$H_*(\SL,\Z)$. }
 \end{remark}

 \begin{remark} \emph{Our methods may  be extended in various natural ways. For example, one
 can take the completed
 cohomology groups $\Htw^*$ with respect to some subset of the primes dividing~$p$ in~$F$  (see~\S~\ref{section:partial}). One 
may also add a tame level
 structure $M$ for $(M,p) = 1$, that is, take the limit over the congruence subgroups~$\Gamma(Mp^r)$. In the latter case,
 the answer will only depend on the radical of~$M$ (that is, the product of distinct primes dividing~$M$), for reasons
 we now explain.
 Note that, by a result of Charney~\cite{Charney}, the cohomology of the congruence subgroup
 $\Gamma(\q^m)$ with coefficients in $\Z_p$ is stable if $(\q,p) = 1$, and
 moreover that the concomitant action of~$\SL(\OL_{\q})$ is trivial. Hence by (an easier version) of the
 argument above, the stable cohomology of $\Gamma(\q^m)$ may be identified with the
 continuous $\Z_p$-cohomology of the homotopy fibre of $\BSL(\OL)^{+}$ mapping
 to $\BSL(\OL/\q^m)^{+}$. By Gabber's rigidity Theorem~\cite{Gabber}, the
 map  $K_n(\OL/\q^m;\Z_p)  \rightarrow  K_n(\OL/\q;\Z_p)$ is an isomorphism,
 and thus, by an application of  Lemma~\ref{lemma:spectral}, the maps
 $H_*(\Gamma(\q^m),\Z_p) \rightarrow H_*(\Gamma(\q),\Z_p)$ are isomorphisms in the
 stable range. (Alternatively, one can simply use the transfer map to see
 that~$H_*(\Gamma(\q),\Z/p^r \Z) \simeq H_*(\Gamma(\q^m),\Z/p^r \Z)_{G(\q)}$ for any~$m \ge 1$ and~$r$
 because~$\Gamma(\q)/\Gamma(\q^m) \simeq G(\q)/G(\q^m)$ has order prime to~$p$.)}
 \end{remark}

 \subsection{Proof of Theorem~\ref{theorem:two}} \label{section:two}
 As explained
 in the the introduction,
 the assumption that~$p$ does not divide~$w_F$ implies, by Theorem~3.6 of~\cite{Serre}, that~$\Htw_1(\SL,\Z/p^r \Z) = 0$ for all~$r$.
 It follows that~$\pi_1(\Yoon;\Z/p^r \Z) = 0$ for all~$r$, and hence, via the Hurewicz map
 (for Hurewicz with coefficients, see~\cite{Coeff}, Thm~9.7), we obtain isomorphisms
 $$\Ktw_2(\OL) :=  \pi_2(\Yoon;\Z_p)
 = \lim_{\leftarrow}  \pi_2(\Yoon;\Z/p^r \Z) = \lim_{\leftarrow}  H_2(\Yoon,\Z/p^r \Z)
  = \Htw_2(\SL,\Z_p),$$
  where the last equality follows from Theorem~\ref{theorem:content}, and
  the second equality was established in the proof of Lemma~\ref{lemma:longlong}.
  The isomorphism
  $$\Htw_2(\SL,\Z_p) = \Ktw_2(\OL) =^{?} \Ktw^{?}_2(\OL) := H^1(G_S,\Q_p/\Z_p(-1))^{\vee}$$
  now follows rationally by Lemma~\ref{lemma:qm},
   and also integrally under the assumption that~$p$ does not divide the order of~$K_2(\OL)$.
   The main statement of Theorem~\ref{theorem:two} is the Pontryagin dual of this equality.
   Now suppose  that $F = \Q$ and $p \ge 3$. Then $|K_2(\Z)| = 2$ and $H^1(G_S,\F_p(-1)) = \F_p$ by
Herbrand's Theorem and class field theory. It follows that
$$\F_p = H^1(G_S,\F_p(-1)) =
H^1(G_S,\Q_p/\Z_p(-1))[p],$$
and thus, by Pontryagin duality, that~$H^1(G_S,\Q_p/\Z_p)^{\vee}/p = \F_p$.
By Lemma~\ref{lemma:q}, 
the group~$H^1(G_S,\Q_p/\Z_p(-1))^{\vee}$  is infinite (it has rank~$d = 1$). Hence,
by Nakayama's Lemma, we deduce that $H^1(G_S,\Q_p/\Z_p(-1))^{\vee} = \Z_p$, and
thus $\Htw_2(\SL,\Z_p) = \Z_p$. \hfill $\qed$

\begin{remark} \emph{ {\bf Homology with coefficients.\rm}  \label{remark:coefficients} Let $\MM_N$ be any algebraic local system (with
$\F_p$ or $\Z_p$ coefficients) for $\SL_N(\OL_F)$.
Then one may also consider the completed homology  groups $\Htw_*(\SL_N,\MM_N) = \lim H_*(\Gamma(p^r),\MM_N)$. 
Since $\MM_N/p^m$ is trivial as a $\Gamma(p^m)$-module, the standard weight--level
argument (Shapiro's Lemma) implies that $\Htw_*(\SL_N,\MM_N) \simeq \Htw_*(\SL_N,\Z_p) \otimes \MM_N$. Hence, given any
 sequence of local systems $\MM_N$ for $\SL_N$ such that~$\lim \MM_N = \MM$, the corresponding sequence
$\Htw_*(\SL_N,\MM_N) = \MM_N \otimes \Htw_*(\SL_N,\Z_p)$ converges to~$\Htw_*(\SL,\Z_p) \otimes \MM$.
}
\end{remark}

\section{Cohomology}

\subsection{The
 Hochschild--Serre Spectral Sequence I} \label{section:star}
 let us consider the completed cohomology groups with coefficients in~$\Q_p/\Z_p$.
 It is easy to see that~$\Htw^0(\SL,\Q_p/\Z_p) = \Q_p/\Z_p$. The congruence subgroup
 property~\cite{Serre} implies that, if~$F$ does not contain any~$p$-th roots of unity,
 then~$\Htw^1(\SL,\Q_p/\Z_p) = 0$ as explained in the introduction.
 However, for any number field~$F$, the group~$\Htw^1(\SL,\Q_p/\Z_p)$
 will always be finite (the obstruction to the full congruence subgroup property is a finite abelian group). 
 For two~$\Z_p$-modules~$A$ and~$B$, let~$A \approx B$ indicate that~$A$ and~$B$
 are isomorphic up
 to a finite group, so~$\Htw^1(\SL,\Q_p/\Z_p) \approx 0$.
 Let $G = \SL(\OL_p)$, and let~$d:=[F:\Q_p]$.
 We have an identification
  $$H^*(G,\Q_p/\Z_p) 
 \approx \Q_p/\Z_p \otimes  \bigotimes_{i=1}^{d} \Lambda_{\Z_p}[x_3,x_5,x_7,\ldots],$$
 where $\Lambda$ denotes the exterior algebra, and the symbol~$\approx$ indicates
 that in each degree we have equality up to a finite group (see Proposition 1 of~\cite{Wagoner}). In particular, the first infinite cohomology
 group in degree bigger than zero is~$H^3(G,\Q_p/\Z_p) \approx (\Q_p/\Z_p)^{d}$.
The Hochschild--Serre spectral sequence for cohomology is the spectral sequence:
$$H^i(G,\Htw^j(\SL,\Q_p/\Z_p)) \Rightarrow H^{i+j}(\Gamma,\Q_p/\Z_p).$$
By Borel's computation of stable cohomology (Theorem~\ref{theorem:borel}), we also have isomorphisms
$$H^1(\Gamma,\Q_p/\Z_p) \approx 0, \quad H^2(\Gamma,\Q_p/\Z_p) \approx 0,  \quad 
H^3(\Gamma,\Q_p/\Z_p) \approx K_3(\OL) \otimes \Q_p/\Z_p.$$
The~$E_2$ page of
the spectral sequence
therefore looks like --- up
to finite groups --- the following:
\begin{center}
{\small
\xymatrix @R=0.6mm @C=0.3cm {
& & \\
 &  & \ar @{-}[ddddddd] \\
\\& q
\\
\\
& 2 & & \Htw^2(\SL,\Q_p/\Z_p) \\
& 1 & & 0 & 0 \quad    & 0  \\
& 0 & & \Q_p/\Z_p &  0 \quad &  0  & (\Q_p/\Z_p)^d  \\
& &   \ar @{-}[rrrrrrr] & & & & & & & &  & \\
& & & 0 & 1\quad    & 2 & 3 & &  p &  \\
& & &
}}
\end{center}
Taking Pontryagin duals and noting
that $\Htw^2(\SL,\Q_p/\Z_p)^{\vee} = \Htw_2(\SL,\Z_p)$, we obtain the following exact sequence (up to finite groups):
$$K_3(\OL) \otimes \Z_p \rightarrow \Z^d_p \rightarrow \Htw_2(\SL,\Z_p)
\rightarrow 0.$$
We deduce:
\begin{corr} \label{corr:bounds} There are inequalities:
$$d = r_1 + 2 r_2 \ge \dim \Htw_2(\SL,\Z_p) \otimes \Q,$$
$$ \dim \Htw_2(\SL,\Z_p) \otimes \Q  \ge r_1 + r_2
= d - \dim K_3(\OL) \otimes \Q.$$
\end{corr}
By Theorem~\ref{theorem:two} (or rather the proof in~\S~\ref{section:two}),  we obtain an isomorphism $\Htw_2(\Z_p) = \Ktw_2(\OL)$,
which therefore identifies~$\Htw_2(\SL,\Z_p)$ with (the dual of) a certain 
Galois cohomology group.
The reader unfamiliar with the complications in computing Galois cohomology groups may be surprised that this equality
does not allow us to (greatly) improve the estimate of Corollary~\ref{corr:bounds}. The difficulty is that determining the ranks of these
groups is effectively a generalization of  Leopoldt's conjecture, which appears to be very difficult.

\subsection{The
 Hochschild--Serre Spectral Sequence II: higher terms}

 Let us suppose that $F = \Q$.
One may  play the spectral sequence game to obtain   information concerning $\Htw^d$ for higher $d$.
Let $G = \SL(\Z_p)$, so
  $$H^*(G,\Q_p/\Z_p) \approx
\Q_p/\Z_p \otimes \Lambda_{\Z_p}[x_3,x_5,x_7,\ldots].$$
 We first consider $\Htw^3$. Let us work (as in the last section) in the Serre category of co-finitely generated $\Z_p$-modules
 up to co-torsion modules (so every term is equivalent to a finite number of copies of $\Q_p/\Z_p$). 
    Since $\Htw^2 \approx \Q_p/\Z_p$, the relevant terms of the spectral sequence are
   \begin{center}
\xymatrix @R=0.6mm @C=0.5cm {
& & \\
 &  & \ar @{-}[dddddddd] \\
\\& q
\\
\\
& 3 & & \Htw^3  \\
& 2 & & \Q_p/\Z_p &  0 &  0\\
& 1 & & 0 & 0 & 0 & 0  \\
& 0 & & \Q_p/\Z_p &  0 &  0 & \Q_p/\Z_p & 0  \\
& &  \ar @{-}[rrrrrrrrr] & & & & & &  & & & & &\\
& & & 0 & 1 & 2 & 3 & 4 & & p  &\\
& & &
}
\end{center}
 Yet $H^3(\Gamma,\Q_p/\Z_p)$ is trivial (since $K_3(\Z) = \Z/48 \Z$ is finite~\cite{LS} --- or alternatively
 we can see this from the computation
 of stable cohomology given by Theorem~\ref{theorem:borel}), and thus $\Htw^3(\SL,\Q_p/\Z_p) \approx 0$.
 Consider the next few terms.   We obtain the following:
 \begin{center}
\xymatrix @R=0.6mm @C=0.5cm {
& & \\
 &  & \ar @{-}[dddddddddd] \\
\\& q
\\
\\
& 5 & & \Htw^5  \\
& 4 & & \Htw^4 & 0 & 0 \\
& 3 & & 0 & 0 & 0 & 0 \\
& 2 & & \Q_p/\Z_p &  0 &  0 & \Q_p/\Z_p & 0 \\
& 1 & & 0 & 0 & 0 & 0 &0 & 0  \\
& 0 & & \Q_p/\Z_p &  0 &  0 & \Q_p/\Z_p & 0 & \Q_p/\Z_p  & 0 \\
& &  \ar @{-}[rrrrrrrrrrr] & & & & & & & & & & & & &\\
& & & 0 & 1 & 2 & 3 & 4 & 5 & 6 &  & p  &\\
& & &
}
\end{center}
 Since $K_4(\Z) \otimes \Q$ is trivial and $K_5(\Z) \otimes \Q$ has rank one, we deduce that
 \emph{either} $\Htw^4(\SL,\Q_p/\Z_p) \approx \Q_p/\Z_p$ and $\Htw^5(\SL,\Q_p/\Z_p) \approx 0$ or
 $\Htw^4(\SL,\Q_p/\Z_p) \approx (\Q_p/\Z_p)^2$ and $\Htw^5(\SL,\Q_p/\Z_p) \approx \Q_p/\Z_p$. In particular, we have that
 $$\dim \Htw_4(\SL,\Z_p) \otimes \Q - \dim \Htw_5(\SL,\Z_p) \otimes \Q = 1.$$
  We can not rule out either possibility,
 just as we cannot rule out that $\Ktw_5(\Z) \otimes \Q \ne 0$ or
 $\Ktw_4(\Z) \otimes \Q \ne \Q_p$; however, assuming the conjectural part of
 Lemma~\ref{lemma:q}, we deduce that $\Htw^4(\SL,\Q_p/\Z_p) \approx \Q_p/\Z_p$ and $\Htw^5 \approx 0$.

 \subsection{Rational completed cohomology groups} 
 We may define rational completed cohomology groups $\Htw^*(\SL,\Q_p)$ as follows. Let  $\Htw^*(\SL,\Z/p^s \Z):= \displaystyle{\lim_{\rightarrow} H^*(\Gamma(p^r),\Z/p^s \Z)}$ and
$$\Htw^*(\SL,\Z_p) = \lim_{\stackrel{s}{\leftarrow}} \lim_{\stackrel{r}{\rightarrow}} H^*(\Gamma(p^r),\Z/p^s\Z)
= \lim_{\leftarrow} \Htw^*(\SL,\Z/p^s \Z).$$
This allows us to define rational completed cohomology groups as follows:
$$\Htw^*(\SL,\Q_p):= \Htw^*(\SL,\Z_p) \otimes \Q_p, \qquad \Htw_*(\SL,\Q_p):=\Htw_*(\SL,\Z_p) \otimes \Q_p.$$
We have (see Theorem~1.1 of~\cite{CEB}) an exact sequence:
$$0 \rightarrow \Hom(\Htw^*[1](\SL,\Z_p)[p^{\infty}],\Q_p/\Z_p)
\rightarrow \Htw^*(\SL,\Z_p) \rightarrow \Hom(\Htw^*(\SL,\Z_p),\Z_p) \rightarrow 0,$$
where $M[1]$ denotes the usual shift of $M$. The modules $\Htw^*(\SL,\Z_p)$ are finitely
generated and so the first term of this sequence is torsion; tensoring with $\Q$ we obtain:
$$\Htw^*(\SL,\Q_p) \simeq \Hom(\Htw_*(\SL,\Z_p),\Q_p).$$
(Wagoner and Milgram considers similar completed cohomology groups when
studying the \emph{continuous} algebraic $K$-theory of local fields, see~\cite{Wagoner}, p.244.)
There is naturally an identification of~$\Htw_*(\SL,\Q_p)$ 
with~$\Hc_*(\Yoon,\Z_p) \otimes \Q_p  = \Hc_*(SK(\OL,\OL_p),\Z_p) \otimes \Q_p$ by Theorem~\ref{theorem:content}.
We may denote the latter group by~$\Hc_*(\Yoon,\Q_p)$.

\subsection{The Eilenberg--Moore spectral sequence} \label{section:EM} A more direct way to compute the
cohomology of the fibre from the cohomology of the total space and the cohomology of the base
is by using the Eilenberg--Moore spectral sequence. This is especially practical in this case since
we know the  cohomology of $\BSL(\Z_p)^{+}$ and $\BSL(\Z)^{+}$ with coefficients in $\Q_p$, namely,
as exterior algebras $\Lambda_{\Q_p}[x_3,x_5,x_7,\ldots]$ and $\Lambda_{\Q_p}[\eta_5,\eta_9,\eta_{13},\ldots]$ respectively. (In the former case, we are taking the continuous cohomology as defined
at the end of the previous section.)
In particular, we take the inverse limits of the Eilenberg--Moore spectral sequences associated
to equation~(\ref{eq:fibration}) 
$$\begin{diagram}
SK(\OL,p^m) & \rTo & SK(\OL) & \rTo &  SK(\OL/p^m)
\end{diagram}$$
with coefficients in $\Z/p^r \Z$
 take the corresponding inverse limit in $r$, and then tensor with $\Q$.
(Relevant here is Proposition~1 of~\cite{Wagoner}.) Here we use the fact that the cohomology of~$SK(\OL,p^m)$
with coefficients in~$\Z/p^r \Z$ is uniformly bounded and that the inverse limit over~$\Z/p^r \Z$ 
recovers the homology groups~$\Htw(\SL,\Z_p)$ by Theorem~\ref{theorem:content}. The uniform 
boundedness of these groups ensures that all inverse limits are Mittag--Leffler and so there are no issues with the derived functor ${\lim}^1$, which will always vanish.
What is less obvious, however, is the structure of $\Lambda_{\Q_p}[\eta_5,\eta_9,\eta_{13},\ldots]$ as a module for
$\Lambda_{\Q_p}[x_3,x_5,x_7,\ldots]$. The natural supposition is  (up to scaling by a non-zero constant depending on $n$) that $x_{4n-1}$
acts as zero and $x_{4n+1}$ acts as $\eta_{4n+1}$. The latter claim, however, is  equivalent to showing that
the $p$-adic regulator map $K_{4n+1}(\Z;\Z_p) \otimes \Q_p \rightarrow K_{4n+1}(\Z_p;\Z_p) \otimes \Q_p$ is an isomorphism.
To see this equivalence, note that~$K(\Z)$ and~$K(\Z_p)$ are infinite loop spaces (and hence~$H$-spaces), and so
the Milnor--Moore theorem (\cite{Moore}, Appendix) identifies the rational classes in~$K$-theory
with the rational primitive classes in homology under the Hurewicz map.
More precisely, for~$K(\Z_p)$, we use the~$p$-adic
Milnor--Moore theorem, which gives the corresponding relationship
between the classes in~$K_*(\Z_p;\Z_p) \otimes \Q$ and the rational
primitive classes in the continuous homology of~$G = \SL(\Z_p)$
(see Proposition 3 of~\emph{ibid} and the subsequent
arguments; indeed,  this is how Wagoner and Milgram computed the groups~$K_*(\Z_p;\Z_p) \otimes \Q$ in the
first place.)
As previously noted, however,
showing that maps  between global and local~$K$-groups are injective is a problem whose difficulty  may be of a similar level to Leopoldt's conjecture (see Conjecture~\ref{conj:L} and the subsequent remarks). 

  One context in which we know this map is an isomorphism is for regular primes, by
Lemma~\ref{lemma:q}. Thus we make the following Ansatz:
\begin{quote}
 $(*)$ \ Either $p$ is regular, or~$F = \Q$ and Conjecture~\ref{conj:L} holds.
 \end{quote}
Under this assumption, the Eilenberg--Moore spectral sequence allows for a complete computation of $\Htw^*(\Q_p)$.
Specifically, we have a spectral sequence:
$$\begin{aligned}
E^{2}_{*,*} = \  & \Tor^{\Lambda_{\Q_p}[x_3,x_5,x_7,\ldots]}_{*,*}(\Q_p,\Lambda_{\Q_p}[\eta_5,\eta_9,\eta_{13},\ldots])    \Rightarrow \Htw^*(\SL,\Q_p) \\
 \simeq  \ &\Tor^{\Lambda_{\Q_p}[x_3,x_7,x_{11},\ldots]}_{*,*}(\Q_p,\Q_p) \Rightarrow \Htw^*(\SL,\Q_p) \\
  \simeq  \ & \bigotimes \Tor^{\Lambda_{\Q_p}[x_{4n-1}]}_{*,*}(\Q_p,\Q_p) \Rightarrow \Htw^*(\SL,\Q_p).
 \end{aligned}$$
This sequence converges because the fundamental group of the base is abelian (since it comes from a $+$ construction).
The first isomorphism requires the assumption that the module structure arises from a surjection of exterior algebras
$\Lambda_{\Q_p}[x_*] \rightarrow \Lambda_{\Q_p}[\eta_*]$, which is exactly the content of assumption $(*)$.
On the other hand, the final sequence degenerates and 
may be computed explicitly;  the limit is given  by a tensor product
of the polynomial algebras on $x_{4n-1}$, shifted by $1$ degree.
In particular, we have the following\footnote{In particular, the rational cohomology conjecturally coincides with the stable rational cohomology of the classical group
$\Sp_{2n}(\Z)$. This is most likely a coincidence; the analogous computation for a general number field $F$ of
signature $(r_1,r_2)$ yields (conditionally on the analogue of $(*)$):
$$\Htw^*(\SL,\Q_p) = \bigotimes_{r_1} \Q_p[x_2,x_6,x_{10},\ldots] \otimes \bigotimes_{r_2} \Q_p[x_2,x_4,x_6,\ldots],$$
which is \emph{not} the stable cohomology of $\Sp_{2n}(\OL_F)$ as  soon as $F$ is not totally real, and nor is
it the stable rational cohomology of any classical group.}:
\begin{theorem} \label{theorem:padicborel}  Let~$F = \Q$, and assume 
that either $p$ is regular or Conjecture~\ref{conj:L} holds.
Then there is an isomorphism $\Htw^*(\SL,\Q_p) \simeq \Q_p[x_2,x_6,x_{10},\ldots]$. In particular, there is an equality
$\dim \Htw_n(\SL,\Z_p) \otimes \Q_p = 0$ unless $n$ is even, and $\dim \Htw_{n}(\SL,\Z_p) \otimes \Q_p$ is the
coefficient of $q^n$ in
$$\prod_{k=1}^{\infty} \frac{1}{(1 - q^{4k-2}) }= \prod_{k=1}^{\infty} (1 + q^{2k})
= 1 + q^2 + q^4 + 2 q^6 + 2 q^8 + 3 q^{10} + 4 q^{12} + 5 q^{14} + 6 q^{16} + 8 q^{18} + \ldots $$ 
\end{theorem}

 \begin{remark} \emph{{\bf Where do the Borel classes go?} The maps
 $H^*(\Gamma,\Q_p) \rightarrow H^*(\Gamma(p^r),\Q_p)$ are isomorphisms (in the
 stable range) for all $r$, and yet, under the assumptions
 of Theorem~\ref{theorem:padicborel},
  $\Htw^*(\SL,\Q_p) = 0$ in all odd degrees; how do
 we reconcile these two statements? The explanation is
 that the classes must become infinitely $p$-divisible up the congruence tower.
 In general, the divisibility of Borel classes in the limit is equivalent to Conjecture~\ref{conj:L},
 by Proposition~\ref{prop:Sch}, this holds for all but finiteiy many of the primitive classes.
  This is in contrast
 to what happens for $\Sp_{2n}(\Z)$. According to Deligne~\cite{Deligne}, the stable class in degree two does 
 not become $p$-divisible for any $p$ up the congruence tower. A back of the envelope calculation suggested by
 the methods of this paper indicates that the same should be true for all stable classes in $\Sp_{2n}(\Z)$. 
 }
 \end{remark}
 
 \begin{remark} \emph{Since~$SK(\OL;\OL_p)$ is an infinite loop space and thus an~$H$-space, one
 way to interpret Theorem~\ref{theorem:padicborel} is simply to note that for~$H$-spaces, that the homology is rationally a polynomial algebra
 on its homotopy. However, one must be slightly careful with such a statement, since our homotopy groups
 are with respect to coefficients in~$\Z_p$, and are homology and cohomology groups with coefficients in~$\Q_p$
 are defined in terms of
  inverse limits (and are thus ``continuous'' (co-)homology groups).  A similar remark applies to the cohomology
  ring of~$\SL(\OL_p)$ and the~$K$-theory of local fields with coefficients in~$\Z_p$. 
  }
  \end{remark}

 \subsection{The
 Hochschild--Serre Spectral Sequence III: determining the differentials}
 We have an isomorphism $H^*(G) \otimes \Q_p = \Lambda[x_3,x_5,x_7,\ldots]$.
We assume in this section that  there is an isomorphism $\Htw^*(\SL,\Q_p) \simeq \Q_p[x_2,x_6,x_{10},\ldots]$
as in~\S\ref{section:EM}. However,
 in order to avoid confusion, we shall use different notation, and in particular we shall write:
$$\Htw^*(\SL,\Q_p) \simeq \Q_p[\y_2,\y_6,\y_{10},\ldots].$$ 
 It follows that the
 second page of the Hochschild--Serre spectral sequence $E^{2}_{*,*}$ is given by
 $$\Q_p[\y_2,\y_6,\y_{10},\ldots] \otimes \Lambda[x_3,x_5,x_7,\ldots]$$
 We may write this out on page two as follows:
 \begin{center}
 {\small
\xymatrix @R=0.4mm @C=0.15cm {
& & \\
 &  & \ar @{-}[ddddddddddddddd] \\
\\& q
\\
\\
& 10 & & \y^5_2, \y^2_2 \y_6, \y_{10} \\
& 9 & & 0 & 0 & \\
& 8 & & \y^4_2, \y_2 \y_6 &  0 &  0 & \\
& 7 & & 0 & 0 & 0 & 0  \\
& 6 & & \y^3_2, \y_6 &  0 &  0 & \y^3_2 x_3, \y_6 x_3 & 0 \\ 
& 5 & & 0 & 0 & 0 & 0  & 0 & 0 & \\
& 4 & & \y^2_2 &  0 &  0 & \y^2_2 x_3 & 0  & \y^2_2 x_5 & 0 \\
& 3 & & 0 & 0 & 0 & 0  & 0 & 0 & 0 & 0   \\
& 2 & & \y_2 &  0 &  0 & \y_2 x_3 & 0  & \y_2 x_5 & 0 & \y_2 x_7 & \y_2 x_3 x_5  \\
& 1 & & 0 & 0 & 0 & 0  & 0 & 0 & 0 & 0  & 0 & 0 &  \\
& 0 & & 1 &  0 &  0 & x_3 & 0  & x_5 & 0 & x_7 & x_3 x_5 & x_9 & x_3 x_7 \\
& &  \ar @{-}[rrrrrrrrrrrrrr] & & & & &  &  &  &   & & & & & & \\
& & & 0 & 1 & 2 & 3 & 4 & 5 & 6 & 7 & 8 & 9 & 10 & &  p  &\\
& & &
}}
\end{center}
The Koszul complex associated to a polynomial algebra induces a spectral sequence 
$$
\begin{aligned}
\Q_p[\y_2,\y_6,\y_{10},\ldots] \otimes \Lambda[x_3,x_5,x_7,\ldots] = & \\
\Lambda[x_5,x_9,x_{13},\ldots]   \otimes \left( \Q_p[\y_2,\y_6,\y_{10} \ldots] \otimes \Lambda[x_3,x_7,x_{11} \ldots]\right)
 & \Rightarrow \Lambda[x_5,x_9,x_{13},\ldots],
\end{aligned}$$
given explicitly by $d^{r} = 0$ unless $r = 4n-1$, in which case $d^{4n-1}(\y_{4n-2}) = x_{4n-1}$ and is zero
otherwise. 
  For example, page $4$ is given as follows:
 \begin{center}
 {\small
\xymatrix @R=0.4mm @C=0.18cm {
& & \\
 &  & \ar @{-}[dddddddddddddddddd] \\
\\& q
\\
\\
& 13 & & 0  &   \\ 
& 12 & & \y^2_6  &  0 & \\ 
& 11 & & 0  &  0 & 0 & \\
& 10 & & \y_{10}  &  0 & 0 & 0 \\
& 9 & & 0 & 0 &    0 & 0 & 0\\
& 8 & &  0 &  0 &  0 &    0 & 0 & 0 \\
& 7 & & 0 & 0 & 0 & 0  &  0 & 0 & 0 \\
& 6 & & \y_6 &  0 &  0 & 0 & 0 &  \y_6 x_5 & 0 & \y_6 x_7  \\ 
& 5 & & 0 & 0 & 0 & 0  & 0 & 0  &  0 & 0 & 0  \\
& 4 & & 0 &  0 &  0 & 0 & 0  & 0 & 0  &  0 & 0 & 0 \\
& 3 & & 0 & 0 & 0 & 0  & 0 & 0 & 0 & 0  &  0 & 0 & 0  \\
& 2 & & 0 &  0 &  0 & 0 & 0  & 0 & 0 & 0 & 0 &  0 & 0 & 0  \\
& 1 & & 0 & 0 & 0 & 0  & 0 & 0 & 0 & 0  & 0 & 0 &  0 & 0 & 0  \\
& 0 & & 1 &  0 &  0 & 0 & 0  & x_5 & 0 & x_7 & 0 & x_9 & 0  & x_{11} & x_5 x_7 & x_{13} & \ldots \\
& &  \ar @{-}[rrrrrrrrrrrrrrrr] & & & & &  &  &  &   & & & & & & &  &\\
& & & 0 & 1 & 2 & 3 & 4 & 5 & 6 & 7 & 8 & 9 & 10 & 11 & 12 & 13 & p  &\\
& & &
}}
\end{center}
Under our assumptions, there is a natural map of spectral sequences from this to
Hochschild--Serre provided that one knows the maps $d^{4n-1}(\y_{4n-2}) = x_{4n-1}$
coincides with the corresponding  maps in Hochschild--Serre (equivalently,
the $y_{4n-2}$ are \emph{transgressive}). By induction, this reduces
to showing that the maps $K_{4n-1}(\Z_p;\Z_p) \rightarrow \Ktw_{4n-2}(\Z,\Z_p;\Z_p)$ are rational isomorphisms,
which we certainly assumed in~\S~\ref{section:EM} in order to compute $\Htw^*(\Q_p)$ in the first place. Thus, by the Zeeman
comparison theorem, these two spectral sequences coincide (\cite{Zeeman2}, see 
also Theorem~3.27 of~\cite{Spectral}).

\section{Classical Cohomology Groups}

We turn in this section to some explicit  computations. Let~$\Gamma_N = \SL_N(\OL_F)$ for
some~$N$ which is sufficiently large so that~$\Htw_*(\SL,\Z_p) = \Htw_*(\SL_N,\Z_p)$
for~$*$ in the range of computation (this is $* \le 3$ except in section~\ref{section:higher}).
Let~$G_N = \SL_N(\Z_p)$.
Benson Farb 
asked (personal communication)
 whether one can compute the
homology groups $H_2(\Gamma(p),\F_p)$ --- we give an complete answer below. 
 The information about such groups is encoded in the completed
cohomology groups $\Htw^*$ together with the differentials of the Hochschild--Serre spectral sequence.
This exercise
probably only serves to indicate why the completed homology groups $\Htw_*$ and completed $K$-groups
$\Ktw_*$  are more natural objects of study than their finite (unstable) analogues.

\medskip

We make the following assumptions on $p$:
\begin{quote}
 $({*}{*})$: \  $p$ does not divide~$w_F$, and  $\Htw^2(\SL,\Q_p/\Z_p) \simeq  H^1(G_S,\Q_p/\Z_p(-1))$.
 \end{quote}

By Theorem~\ref{theorem:two}, this holds if $p$ does not divide $w_F |K_2(\OL_F)|$, and 
  conjecturally  always holds for $p > 2$. In fact, we shall state the theorems below with
  the stronger hypothesis $p \nmid w_F |K_2(\OL_F)|$, but we only use the assumption above.
  Let $d = [F:\Q]$.
Because~$p$ does not divide~$w_F$, we have $\Htw^1(\SL_N,\Q_p/\Z_p) = 0$ (once more
using~\cite{Serre}),
and thus, from the long exact sequence of completed
cohomology corresponding to the short exact sequence of modules:
$$0 \rightarrow \F_p \rightarrow \Q_p/\Z_p \stackrel{\times p}{\longrightarrow}\Q_p/\Z_p
\rightarrow 0,$$
we find that
$$\Htw^2(\SL_N,\F_p) = \Htw^2(\SL_N,\Q_p/\Z_p)[p] = H^1(G_S,\Q_p/\Z_p(-1))[p] 
\simeq H^1(G_S,\F_p(-1)).$$
The Hochschild--Serre spectral sequence
$H^i(G_N(p^m),\Htw^j(\SL_N,\F_p)) \Rightarrow H^{i+j}(\Gamma_N(p^m),\F_p)$ gives rise in the usual way
(using $\Htw^1(\SL_N,\F_p) = 0$) to
the exact sequence:
{\small
               $$0 \rightarrow H^2(G_N(p^m),\F_p)
\rightarrow H^2(\Gamma_N(p^m),\F_p) \rightarrow H^1(G_S,\F_p(-1))
\rightarrow H^3(G_N(p^m),\F_p) \rightarrow H^3(\Gamma_N(p^m),\F_p)
$$
}
Consider  the cohomology ring $H^*(G_N(p^m),\F_p)$. Since $p > 2$, we deduce that
the group $G_N(p^m)$ is $p$-powerful, and (by a theorem of Lazard) that
$G_N(p^m)$
 is
a Poincar\'{e} homology group of dimension $\dim(G_N)= d(N^2 - 1)$. 
More explicitly, by  Lazard~\cite{Lazard} (Ch V, 2.2.6.3 and~2.2.7.2, p.167), we deduce that
the cohomology of $G_N(p^m)$ is an exterior algebra on the degree one classes:
$$H^*(G_N(p^m),\F_p) \simeq \wedge^*(H^1(G_N(p^m),\F_p)).$$
The group $G_N/G_N(p^m) =   \SL_N(\OL/p^m)$  acts naturally on this ring, the
action factors through $\PSL_N(\OL/p)$. Let us denote $H^1(G_N(p^m),\F_p)$ by
$\LL$ as an $\SL_N(\OL/p)$-module. We have
$$\LL \simeq   \Hom(H_1(G_N(p^m),\F_p),\F_p) \simeq
\Hom(M^0_N(\F_p),\F_p),$$
where $M^0_N(\F_p)$  denotes the matrices of trace zero. 
We deduce
 that there is an exact sequence:
\begin{equation}
\label{eq:deduce}
0 \rightarrow \wedge^2 \LL \rightarrow H^2(\Gamma_N(p^m),\F_p) \rightarrow 
H^1(G_S,\F_p(-1))   \rightarrow \wedge^3 \LL. \end{equation}
Since the map $H^3(G_N(p^m),\F_p) \rightarrow H^3(G_N(p^{m+1}),\F_p)$ is zero for all $m \ge 1$,
we also deduce that:
\begin{lemma} Suppose that~$p$ does not divide~$w_F |K_2(\OL_F)|$. \label{lemma:one} \label{LEMMA:ONE} For $m \ge 1$, there is an exact sequence:
$$0 \rightarrow \wedge^2 \LL \rightarrow H^2(\Gamma_N(p^m),\F_p) \rightarrow 
H^1(G_S,\F_p(-1))   \rightarrow 0.$$
For $m \ge 2$, this sequence splits.
\end{lemma}

\begin{proof}
By the computation above, we have exact sequences and commutative diagrams as follows:
{\small
$$\begin{diagram}
0  & \rTo &  \wedge^2 \LL  & \rTo & H^2(\Gamma_N(p),\F_p) & \rTo &
H^1(G_S,\F_p(-1))  & \rTo &  \wedge^3 \LL \\
 & & \dTo & & \dTo & & \dEquals & & \dTo \\
0  & \rTo &  \wedge^2 \LL  & \rTo & H^2(\Gamma_N(p^m),\F_p) & \rTo &
H^1(G_S,\F_p(-1))  & \rTo &  \wedge^3 \LL
 \end{diagram}$$
 }
 where~$m > 1$. The equality in the third column follows from the fact that
 the invariants of~$\Htw^2(\SL_N,\F_p)$ under~$G_N(p^m)$ do not depend on~$m$, because
 the entire group~$G_N$ acts trivially.
 On the other hand, the maps on the first and last columns are identically zero, because
 they are induced from the map of groups
 $$G_N(p)/G_N(p^2) 
 \rightarrow G_N(p^m)/G_N(p^{m+1}),$$
 which is the zero map.
In particular, the vanishing of the map in the last
column implies by commutativity that the image of~$H^1(G_S,\F_p(-1))$ in~$\wedge^3 \LL$ must be trivial, which
shows that the sequence of the lemma
 is exact. On the other hand, the vanishing of the map in the first
row implies that the quotient~$H^2(\Gamma_N(p),\F_p)/\wedge^2 \LL = H^1(G_S,\F_p(-1))$ maps into~$H^2(\Gamma_N(p^m),\F_p)$
in a  way that is compatible with the map to~$H^1(G_S,\F_p(-1))$, which splits the sequence for~$m > 1$.
\end{proof}

\subsection{A mild improvement of Lemma~\ref{lemma:one}}

\begin{lemma} \label{prop:benson} 
Suppose  that $p > 3$ splits completely in $F$ and does not divide the order of
$K_2(\OL_F)$ or $K_3(\OL_F)$.
Then, for  all $m$, there is an isomorphism
$$ H^2(\Gamma_N(p^m),\F_p)  \simeq \wedge^2 \LL \oplus (\F_p)^{d}.$$
\end{lemma}

\begin{remark} \emph{The condition on $p$ is never satisfied unless $F$ is totally real,
since otherwise $K_3(\OL_F)$ is infinite. If $F$ is totally real, however, then the assumptions
hold for all but finitely many $p$ that split completely in $F$. For example, if $F = \Q$, then the assumptions hold for $p > 3$.
 The splitting condition should not be necessary --- we assume it in order to invoke a theorem of~\cite{FriedCompute} concerning
 the homology of $\GL_N(\F_p)$ of some particular module;
  the methods of~\cite{FriedCompute} should also apply to $\GL_N(\F_q)$ for a finite field $\F_q$, which would be
  sufficient provided that $p$ is unramified in $F$.}
\end{remark}

\begin{proof} The assumptions on $p$ imply that
$H^3(\Gamma_N,\F_p) = H^2(\Gamma_N,\F_p) = 0$, and hence that $\Htw^2(\SL_N,\F_p)
\simeq H^3(G_N,\F_p)
\simeq (\F_p)^{d}$. Consider the  \emph{classical} Hochschild--Serre spectral sequence
 $$H^i(\SL_N(\OL/p),H^j(\Gamma_N(p),\F_p)) \Rightarrow H^{i+j}(\Gamma_N,\F_p).$$
 As representations for $\SL_N(\OL/p)$, we have
 $H^1(\Gamma_N(p),\F_p) = \LL$ and $H^0(\Gamma_N(p),\F_p) = \F_p$.
 By Quillen's computation of~$K_*(\F_q)$ in~\cite{Quillen} (Theorem~$6$),
 if $p$ is unramified in $F$, then $K_n(\OL/p) \otimes \F_p  = 0$ for $n > 0$, or equivalently that
 $H^n(\SL_N(\OL/p),\F_p) = 0$ for all $n > 0$.
 Since  $\LL$ has no $\SL_N(\OL/p)$-invariants,
 the second page of the sequence is as follows:
 \begin{center}
 {\small
\xymatrix @R=0.6mm @C=0.25cm {
& & \\
 &  & \ar @{-}[ddddddd] \\
\\& q
\\
\\
& 2 & &  H^0(\SL_N(\OL/p),H^2(\Gamma_N(p),\F_p)) & \\
& 1 & & 0 & H^1(\SL_N(\OL/p),\LL) & H^2(\SL_N(\OL/p),\LL)  \\
& 0 & & \F_p & 0 & 0&  0 \quad &  0  \\
& &  \ar @{-}[rrrrrr] & & & & & &  & & & &\\
& & & 0 & 1 & 2 & 3 \quad  & 4   &\\
& & &
}}
\end{center}
Since $H^2(\Gamma_N,\F_p) = 0$ and $H^3(\Gamma_N,\F_p) = 0$ by assumption, we deduce that there is an isomorphism:
$$H^0(\SL_N(\OL/p),H^2(\Gamma_N(p),\F_p)) \simeq H^2(\SL_N(\OL/p),\LL).$$
By Proposition~3.0 of~\cite{FriedCompute}, there is an isomorphism:
$H_2(\GL_N(\F_p),\LL^{*}) \simeq \Z/p\Z$, which implies that 
$H^2(\GL_N(\F_p),\LL) = \Z/p\Z$, and hence  
that $H^2(\SL_N(\F_p),\LL)$ has dimension
\emph{at least} one.  Explicitly, since~$\GL_N(\F_p)/\SL_N(\F_p) = \F^{\times}_p$ has order prime to~$p$, the
Hochschild--Serre spectral sequence degenerates and one has an isomorphism:
$$H^*(\GL_N(\F_p),\LL) = H^*(\SL_N(\F_p),\LL)^{\F^{\times}_p},$$
and so a lower bound for the former group gives a lower bound for the latter.
Since $p$ is totally split in $F$, 
we have $\SL_N(\OL/p) = \prod_{v|p} \SL_N(\F_p)$, and thus  
 $H^2(\SL_N(\OL/p),\LL)$ has dimension at least $[F:\Q]$. 
Since $H^0(\SL_N(\OL/p),\wedge^2 \LL) = 0$, 
from the exact sequence~\ref{eq:deduce}, we deduce that there is an injection:
$$H^0(\SL_N(\OL/p),H^2(\Gamma_N(p),\F_p)) \rightarrow H^1(G_S,\F_p(-1)) \simeq (\F_p)^{[F:\Q]}.$$
The latter equality relies on the fact that~$p$ does not divide~$w_F |K_2(\OL_F)|$;
howeber, we are assuming that~$p$ does not divide the order of~$K_2(\OL_F)$, and~$p > 2$
cannot divide~$w_F$ because~$F$ is totally real.
As we just proved
the left hand side has dimension at least $[F:\Q]$, this map is an isomorphism, which therefore splits~(\ref{eq:deduce}), proving the lemma.
\end{proof}

This allows us to answer the question of Farb:

\begin{corr} \label{corr:benson}   Let $p > 3$, and let $\Gamma_N(p)$ denote
the congruence subgroup of $\SL_N(\Z)$. Then for sufficiently large $N$, there are isomorphisms 
of $\SL_N(\F_p)$-modules:
$$H^2(\Gamma_N(p),\F_p) \simeq \wedge^2 \LL \oplus \F_p, \qquad
H_2(\Gamma_N(p),\F_p) \simeq \wedge^2 \LL^* \oplus \F_p,$$
where $\LL^*$  is isomorphic to 
 $M^0_N(\F_p)$.
In particular,
$\displaystyle{\dim H^2(\Gamma_N(p),\F_p) = \binom{N^2 - 1}{2} + 1}$.
\end{corr}

\subsection{Remarks on cohomology in higher degrees} \label{section:higher}
One may continue computations as above in higher degrees, although the analysis becomes more
and more intricate.
Consider, for example,  the group $H^3(\Gamma_N(p^m),\F_p)$ when~$F = \Q$. 
For $p > 3$, we should have~$\Htw_3(\F_p) = 0$. The
Hochschild--Serre spectral sequence $H_i(G_N(p^m),\Htw_j(\SL_N,\F_p)) \Rightarrow
H_{i+j}(\Gamma_N(p^m),\F_p)$ would then yield:
 \begin{center}
 {\small
\xymatrix @R=0.6mm @C=0.25cm {
& & \\
 &  & \ar @{-}[dddddddd] \\
\\& q
\\
\\
& 3 & & 0 & 0 \\
& 2 & &  \F_p &  \LL & \wedge^2 \LL&   \\
& 1 & & 0 & 0 &  0 & 0  \\
& 0 & & \F_p & \LL & \wedge^2 \LL&  \wedge^3 \LL & \wedge^4 \LL \\
& &  \ar @{-}[rrrrrr] & & & & & &  & & & &\\
& & & 0 & 1 & 2 & 3   & 4   &\\
& & &
}}
\end{center}
Since $H^2(\Gamma_N(p^m),\F_p) = \wedge^2 \LL \oplus \F_p$, this would give an exact sequence:
$$0 \rightarrow \wedge^3 \LL \rightarrow H^3(\Gamma_N(p^m),\F_p) \rightarrow \LL \rightarrow
\wedge^4 \LL.$$
In this case, the map $H^3(\Gamma_N(p^m),\F_p) \rightarrow H^3(\Gamma_N(p^{m+1}),\F_p)$ is zero.
Hence the order of $H^3(\Gamma_N(p^m),\F_p)$ should be the order of $\wedge^3 \LL$, up to
an error that is bounded by the order of $\LL$. 
More generally, we have:
\begin{lemma}  For $F = \Q$, all primes~$p$, and all sufficiently large~$N$, the natural edge map:
$$\wedge^k \LL = H^k(G_N(p^m),\F_p) \rightarrow H^k(\Gamma_N(p^m),\F_p)$$
has kernel and cokernel whose dimensions  are $O(N^{2(k-2)})$, where the implied constant does
not depend on $N$. In particular,
$$\dim H^k(\Gamma_N(p^m),\F_p) = \binom{N^2-1}{k} +O(N^{2(k-2)})
= \frac{N^{2k}}{k!}- \binom{k+1}{2}  \frac{N^{2k-2}}{k!} + O(N^{2k-4})$$
\end{lemma}

\begin{proof} This follows from the the Hochschild--Serre spectral sequence, noting that
the first row is zero (since $\Htw^1(\SL_N,\F_p) = 0$).
\end{proof}

\begin{remark} \emph{It follows from Theorem~1.5 of~\cite{CFN} that,
for all~$F$ and~$p$, the quantity $\dim H^k(\Gamma_N(p^m),\F_p)$ (in the stable range)
is actually a \emph{polynomial} in $N$. 
}
\end{remark}

\begin{remark} \emph{One obtains a corresponding result for number fields, except $\Htw^1(\SL_N,\F_p)$
does not vanish in general if $F$ contains $p$th roots of unity; and thus one only obtains the estimate
(without any assumption on~$p$ or~$F$):
$$\dim H^k(\Gamma_N(p^m),\F_p) = \frac{N^{2kd}}{k!} + O(N^{2 k d -2}).$$
}
\end{remark}

\begin{remark}
\emph{
One may sensibly define the stable homology groups $\Hs_*(\Gamma_N(p^m),\Z_p)$
and  $\Hs_*(\Gamma_N(p^m),\F_p)$ for sufficiently large~$N$ to be the images
of $\Htw_*(\SL_N,\Z_p)$ and $\Htw_*(\SL_N,\F_p)$  in $H_*(\Gamma_N(p^m),\Z_p)$ and $H_*(\Gamma_N(p^m),\F_p)$ respectively. Our results certainly imply that the action
of $G_N/G_N(p^m)$ on $\Hs_*(\Gamma_N(p^m),\Z_p)$ is trivial in the stable range. Moreover, one expects that these groups
should be directly related (perhaps even equal) to the continuous homology groups
$\Hc_*(\Ymn,\Z_p)$, which then relate to the $K$-theory (with coefficients) of the rings $\Z/p^m\Z$.
Indeed, one can see a reflection of the calculation of these groups (in the simplest situations) in the computations above.
}
\end{remark}

\section{Partially Completed \texorpdfstring{$K$}{K}-groups}
\label{section:partial}

Suppose that $F$ is an number field, and that $\p | p$ is a prime in $\OL_F$.  Then one can apply the analysis
above to the partially completed homology groups
$$\Htw_{*}(\SL_N,\p,\Z_p) := \lim_{\leftarrow} H_*(\Gamma_N(\p^r),\Z_p).$$
(More generally, one can do this for any set of primes dividing $p$.)
The arguments of~\cite{CEA} show that these groups do not depend on~$N$ for sufficiently large~$N$
(and have a trivial action of~$\SL_N(\OL_{\p})$)
and for such~$N$ we write~$\Htw_*(\SL,\p,\Z_p)$.
Although we no longer have recourse to the Poitou--Tate sequence (which requires
that $S$ contain all places above $p$),
it turns out that these groups are even simpler to understand than the standard
completed cohomology groups  under the following favorable circumstances.
In particular, we shall find contexts in which all the partially completed homology groups (beyond~$\Htw_0$)
vanish identically.

\medskip

Let $F/\Q$ be an imaginary quadratic field. Suppose that $p$ splits in $F$, and let $v|p$
be a place above $p$. Let $\GammaG$ be the Galois group of the maximal 
pro-$p$ extension of $F(\zeta_p)$ over $F$
unramified outside the places $S$ above $p$. Let $D_v$ be the decomposition group at $v$.

\begin{df}  \label{df:veryregular} Let $p$ be a prime that splits in an imaginary quadratic field $F/\Q$.
The prime $p$ is {\bf very regular\rm} if the map:
$$\Gal(\Qbar_p/\Q_p) \rightarrow D_v \subseteq \GammaG$$
is surjective for either $v|p$.
\end{df}

If the result is true for one~$v|p$, then it is also true for the other (the images of~$D_v$ for~$v|p$ are 
permuted by the action of~$\Gal(F/\Q)$.)

 \begin{lemma} \label{lemma:equals} Suppose that $p$ is very regular. Then the  map $H^i(G_S,M) \rightarrow H^i(D_v,M)$ is an isomorphism for all $i$ and all $G_S$-modules $M$ that are sub-quotients
 of the Tate twists $\Z_p(n)$ for any $n \in \Z$.
  \end{lemma}
  
 Before proving this lemma, we note the following consequence, which is the reason for considering 
 very regular primes.
 
 \begin{corr}  \label{corr:replace} If~$p$ is a very regular prime, then the map~$K_n(\OL_F) \otimes \Z_p \rightarrow
 K_n(\OL_{\p};\Z_p)$ is an isomorphism for~$n > 1$.
 \end{corr}
 
This follows from the description of these groups in terms of Galois cohomology (Theorem~\ref{theorem:V})
and the fact that the corresponding map on Galois cohomology is an isomorphism.
  
\begin{proof}[Proof of Lemma~\ref{lemma:equals}] For such modules~$M$ as in the statement of the lemma, there is a canonical isomorphism $H^i(G_S,M)
\simeq H^i(\GammaG,M)$, because the cohomology groups will only depend on the pro-$p$
completion of the fixed field of any relevant module. 
There is a  natural isomorphism
$H^0(G_S,M) = H^0(D_v,M_v)$, 
and the map:
$$H^1(G_S,M) \rightarrow H^1(D_v,M_v),$$
is injective by inflation--restriction. 
 Let $\Sigma$ denote the set of Selmer conditions where no condition is imposed
 at $v|p$, and the dual Selmer condition $\Sigma^*$ consists of classes
 that are totally trivial locally at $v|p$, and the classes are unramified outside~$p$. We have
$H^1_{\Sigma}(F,M) = H^1(G_S,M)$ by definition, where~$H^1_{\Sigma}(F,M)$ as usual
denotes classes in~$H^1(F,M)$ which satisfy the local conditions corresponding to~$\Sigma$. Since $p$ is very regular,
the group $H^1_{\Sigma^*}(F,M^*) = 0$ because any such class is trivial in $H^1(D_v,M^*)$.
To this point, we have  not used the fact that $F$ is
an imaginary quadratic field. This condition  arises in the numerical computation
of Selmer groups via the  formula of Greenberg--Wiles~\cite{Gre,W,Fermat} and the global Euler characteristic
formula~\cite{Milne,TateDual}. From the Greenberg--Wiles formula, we have
$$|H^1(G_S,M)| = \frac{|H^0(F,M)|}{|H^0(F,M^*)|} \cdot \frac{|H^1(D_v,M)|^2}{|H^0(D_v,M)|^2}
 \cdot \frac{1}{|M|}.$$
 From the local Euler characteristic formula,
 $$|H^1(D_v,M)| = |H^0(D_v,M)| \cdot |H^2(D_v,M)| 
 \cdot |M| =  |H^0(D_v,M)| \cdot |H^0(D_v,M^*)| \cdot |M|,$$
 and thus
$$\begin{aligned} |H^1(G_S,M)| = & \ \frac{|H^0(F,M)|}{|H^0(F,M^*)|} \cdot \frac{|H^1(D_v,M)|^2}{|H^0(D_v,M)|^2}
 \cdot \frac{1}{|M|} \\
 = & \ \frac{|H^0(F,M)|}{|H^0(F,M^*)|} \cdot \frac{|H^1(D_v,M)|}{|H^0(D_v,M)|} \cdot
  \frac{ |H^0(D_v,M)| \cdot |H^0(D_v,M^*)| \cdot |M|}{|H^0(D_v,M)|}
 \cdot \frac{1}{|M|} \\
  = & \ \frac{|H^0(F,M)|}{|H^0(F,M^*)|} \cdot |H^1(D_v,M)| \cdot \frac{1}{|H^0(D_v,M)|} \cdot 
  |H^0(D_v,M^*)| \\
 = & \ \frac{|H^0(F,M)|}{|H^0(D_v,M)|} \cdot \frac{|H^0(D_v,M^*)|}{|H^0(F,M^*)|} \cdot |H^1(D_v,M)|
 \\
  = &\  |H^1(D_v,M)|.
  \end{aligned}$$
  It follows  that $H^1(G_S,M) \rightarrow H^1(D_v,M)$, which we already showed was an injection, is actually an isomorphism.
  On the other hand, by the global and local Euler characteristic formulae,
  $$\begin{aligned}
  |H^2(G_S,M)| = & \ \frac{|H^1(G_S,M)|}{|H^0(G_S,M)|} \cdot \frac{H^0(G_{\C},M)}{|M|^2} \\
  = & \ \frac{|H^1(D_v,M)|}{|H^0(D_v,M)|}  \cdot \frac{1}{|M|}
  = |H^2(D_v,M)|. \end{aligned}$$
  By the $5$-lemma and d\'{e}vissage, to prove the lemma 
  it suffices to show that the maps 
  $$H^2(G_S,\F_p(n)) \rightarrow H^2(D_v,\F_p(n))$$
   are
  injective (equivalently, isomorphisms) for any $n \in \Z$ (we use here the fact that the local and global
  cohomology groups
  have cohomological dimension~$2$ in this context). 
  The only non-trivial case (i.e., when both groups are not trivial) is $n = 1$. In this case, there is a commutative diagram:
  $$
  \begin{diagram}
  H^1(G_S,V) & \rTo & H^1(G_S,\F_p) & \rTo & H^1(G_S,\F_p(1)) \\
  \dEquals & & \dEquals & & \dTo \\
  H^1(D_v,V) & \rTo & H^1(D_v,\F_p) & \rTo & H^1(D_v,\F_p(1)) \\
  \end{diagram}
  $$
 where $V$ is any non-split extension of $\F_p$ by $\F_p(1)$.
 Since $V$ is non-split, $H^0(V^*) = 0$, and hence $H^2(V) = 0$.
Thus the maps $H^1(\F_p) \rightarrow H^1(\F_p(1))$ are surjective,
 and it follows (by commutivity) that map $H^1(G_S,\F_p(1)) \rightarrow H^1(D_v,\F_p(1))$ must be
 surjective, hence 
 an isomorphism.
 \end{proof}

We now give a numerical characterization of very regular primes.

\begin{lemma} \label{lemma:crit} Let $F$ be an imaginary quadratic field such that $p > 2$ splits completely,
and let $v|p$ be a place above $p$. Let $w$ be the unique place above $v$ in $F(\zeta_p)$.
Suppose that:
\begin{enumerate}
\item  The maximal exponent $p$ abelian extension of $F(\zeta_p)$ that is unramified outside $w$
is cyclic.
\item If $p = \p \pbar$ in $\OL_F$ and $h = |\Cl(\OL_F)|$, then
the projection of $\p^h$ generates $(1 + \pbar)/(1 + \pbar^2)$.
\end{enumerate}
Then $p$ is very regular. Conversely, if either of these conditions fail, then $p$ is not very regular.
\end{lemma}

\begin{proof} To show the map of pro-$p$ groups $D_w \rightarrow \GammaG$ is surjective,
it suffices to prove that $D_w$ surjects onto the Frattini quotient of $\GammaG$. The cokernel
of this map corresponds to an abelian exponent $p$ extension of $F(\zeta_p)$ that is unramified at $w$ and at all primes away from $p$.
If the cokernel is not cyclic, then there exists a quotient where $w$ splits completely, and
$p$ is not very regular. Let us assume it is cyclic.
By class field theory,
$F$ itself admits an abelian $p$-extension $H/F$ unramified at $v$ and all primes away from $p$, and hence the extension over
$F(\zeta_p)$ descends to $F$.  It suffices to show that this extension is inert at $v$ if  and only if
the second condition holds. Equivalently, it suffices to show that the maximal abelian 
$p$-extension
of $F$ completely split at $v|p$ and unramified outside $p$ is trivial.
Let $\p$ and $\pbar$ be the primes in $\OL_F$ corresponding to the place $v|p$ and its conjugate. Recall that the $p$-part of the ray class group of conductor $\pbar^2$ lives in an 
exact sequence:
$$\OL^{\times}_F \cap (1 + \pbar)  \rightarrow (1 + \pbar)/(1 + \pbar^2) \rightarrow
\RCl(\pbar^2) \otimes \Z_p \rightarrow \Cl(\OL_F) \otimes \Z_p \rightarrow 0.$$
Since~$F$ is an imaginary quadratic field,~$\OL^{\times}_F$
consists entirely of roots of unity, and so the first term~$\OL^{\times}_F \cap (1 + \pbar)$ vanishes because
we are assuming that~$p$ is prime to~$w_F$ (this follows because~$p > 2$ is unramified in~$F$).
It suffices to show that $\RCl(\pbar^2) \otimes \Z_p$ is generated by $\p$,
since, by Nakayama, any generator of $\RCl(\pbar^2) \otimes \Z_p$
lifts to a generator of $\RCl(\pbar^m) \otimes \Z_p$ for any $m$. Since $\RCl(\pbar^2) \otimes \Z_p$ is an abelian
$p$-group of order $p \cdot h_p$ (where $h_p$ is the $p$-part of $h$), it is cyclic with generator $\p$ if and only if 
$\p^{h}$ has order $p$. Yet $\p^h$ is non-trivial in $\RCl(\pbar^2) \otimes \Z_p$
if and only if its projection to $ (1 + \pbar)/(1 + \pbar^2)$ is non-trivial.
\end{proof}

Note that the projection of any principal ideal $(\alpha)$ onto $(1 + \pbar)/(1 + \pbar^2)$
is given by the image of $\alpha^{p-1}$.

We give an alternate numerical  formulation of the previous lemma:

\begin{lemma} \label{lemma:crit2}
Let $F$ be an imaginary quadratic field such that $p > 2$ splits completely,
and let $v|p$ be a place above $p$. Let $\chi$ be the odd quadratic character corresponding
to $F$. Then $p$ is very regular if and only if the following hold:
\begin{enumerate}
\item $p$ is a regular prime, that is, $p > 3$ does not divide $\zeta(-1)$, $\zeta(-3)$, $\zeta(-5), \ldots,
\zeta(4-p)$.
\item $p$ does not divide any of the $L$-values $L(\chi,-2)$, $L(\chi,-4)$, $L(\chi,-6), \ldots, L(\chi,3-p)$.
\item If $p = \p \pbar$ in $\OL_F$ and $h = |\Cl(\OL_F)|$, then
the projection of $\p^h$ generates $(1 + \pbar)/(1 + \pbar^2)$.
\end{enumerate}
\end{lemma}

\begin{proof} It suffices to show that the conditions are equivalent
to those of Lemma~\ref{lemma:crit}.  If either $p$ is not regular or divides one of the $L$-values, then,
by the main conjecture of Iwasawa theory~\cite{WilesMain},
there exists an unramified extension of $\F_p$ by either $\F_p(n)$ or $\F_p(n) \otimes \chi$
for some $n \ne 0,1 \mod (p-1)$. Since $H^1(\Q_p,\F_p(n)) = 0$ for such $n$, these extensions necessarily
split completely at primes above $p$, implying that $p$ is not very regular.  By 
Lemma~\ref{lemma:crit}  and its proof the same is true of the final condition.

Now suppose that
 $p$ is not very regular. Equivalently,  the subspace of $H^1(F(\zeta_p),\F_p)$ consisting
of classes that are trivial  away from $p$ and for some $w$ with $w|p$ is non-zero. The  group $\Gal(F(\zeta_p)/F)$ acts on this
space, and the third condition above implies that the projection of this space onto
$H^1(F,\F_p)$ is trivial. This implies the same for $H^1(F,\F_p(1))$ by
the Greenberg--Wiles formula.
Hence there exists a non-trivial $n \ne 0,1 \mod (p-1)$ such that
$$H^1(G_{F,S},\F_p(n)) \simeq H^1(G_{\Q,S},\F_p(n)) \oplus H^1(G_{\Q,S},\F_p(n) \otimes \chi)$$
contains a class $\eta$ that is unramified at $v|p$. 
Suppose that  $\dim H^1(G_{F,S},\F_p(n)) = 1$.  Then
$\eta \in H^1(\Q,V)$ where $V$ is either $\F_p(n)$ or $\F_p(n) \otimes \chi$.
As these spaces are $\Gal(F/\Q)$-invariant, it follows that $\eta$ is
also  unramified at the other prime dividing~$p$. Now either:
\begin{enumerate}
\item Complex conjugation acts on $V$ by $-1$, in which
case by the main conjecture one obtains a divisibility of $L$-values as above, or
\item Complex conjugation acts by $+1$ on $V$, in which case, by
the Greenberg--Wiles formula, there exists an everywhere unramified class in $H^1(\Q,V^*)$,
that also implies a divisibly of $L$-values (again by the main conjecture).
\end{enumerate}
Suppose now that $\dim H^1(G_{F,S},\F_p(n)) > 1$.  If either 
$H^1(G_{\Q,S},\F_p(n))$ or $H^1(G_{\Q,S},\F_p(n) \otimes \chi))$ is zero, the same
argument as above applies. Hence we may assume that
$H^1(G_{\Q,S},V) \ne 0$ where complex conjugation acts by $1$ on $V$, and the
result follows as in the previous case.
\end{proof}
 
 The following result shows that the partially completed
homology groups at very regular primes are particularly simple:

\begin{lemma}   Suppose that $p$ is very regular in a imaginary quadratic field $F/\Q$.
Let $\p$ denote a prime above $p$. Then the stable $\p$-completed homology groups:
$$\Htw_{n}(\SL,\p,\Z_p) = \lim_{\leftarrow} H_n(\SL,\Gamma(\p^r),\Z_p)$$
are trivial for all $n > 0$, and equal to $\Z_p$ for $n = 0$.
\end{lemma}

\begin{proof}  The analogues of the completed~$K$-groups in these
contexts are the homotopy groups with coefficients in~$\Z_p$
of the homotopy fibres~$K(\OL,\OL_{\p};\Z_p)$ and $SK(\OL,\OL_{\p};\Z_p)$.
Let
$$\Ktw_{n}(\OL,\p):= \Ktw_{n}(\OL,\OL_{\p};\Z_p):=\pi_n(K(\OL,\OL_{\p});\Z_p)$$
$$\SKtw_{n}(\OL,\p):= \SKtw_{n}(\OL,\OL_{\p};\Z_p):=\pi_n(SK(\OL,\OL_{\p});\Z_p)$$
 for~$n > 0$. 
By Corollary~\ref{corr:replace},
the natural map $SK_n(\OL_F;\Z_p) \rightarrow SK_n(\OL_{\p};\Z_p)
= SK_n(\Z_p;\Z_p)$  is an isomorphism for all $n \ge 2$, and (because both vanish) the map
is  also an isomorphism for~$n = 1$.  Hence, from the
Serre exact sequence, the completed homotopy
 groups $\SKtw_n(\OL,\p)$ vanish for all $n$. The result follows by the Hurewicz theorem.
 \end{proof}
 
This lemma allows us to explicitly compute all the classical congruence homology groups
at powers of a very regular prime in the stable range.
 
 \begin{theorem}  \label{theorem:exp}   Suppose that $p$ is very regular in a imaginary quadratic field $F/\Q$, and fix an integer $n$.
Let $\p$ denote a prime above $p$,  let $\Gamma_N = \SL_N(\OL_F)$,
 let $G_N =  \SL_N(\OL_{\p}) = \SL_N(\Z_p)$. 
Then, for all sufficiently large $N$, there are, for all positive integers $m$,
 isomorphisms 
 $$
H_n(\Gamma_N(\p^m),\Z_p) \simeq H_n(G_N(p^m),\Z_p),$$
 $$H_n(\Gamma_N(\p^m),\F_p) \simeq H_n(G_N(p^m),\F_p) \simeq
\wedge^n M^0_N(\F_p),$$
where $M^0_N(\F_p)$ denotes the trace zero $N \times N$ matrices with coefficients in~$\F_p$.
Moreover, these isomorphisms respect the $G_N/G_N(\p^m)$ and $G_N/G_N(\p) = \SL_N(\F_p)$-module
structures respectively.
\end{theorem}

\begin{proof} This follows immediately from the Hochschild--Serre spectral sequence,
since $E^{2}_{ij} = 0$ unless $i = 0$, and hence the sequence degenerates immediately.
Note that the vanishing of $\Htw_*(\SL_N,\p,\Z_p)$ implies the vanishing of $\Htw_*(\SL_N,\p,\F_p)$.
\end{proof}

\begin{remark} \emph{Algebraic local systems of $\GL$ for an imaginary quadratic field
are algebraic representations of $\GL_N(\C)$ as a real group, in particular, they are
direct sums of representations of the form
$V_{\mu} \otimes \overline{V_{\lambda}}$, where $V_{\mu}$ and $V_{\lambda}$ are algebraic representations of $\GL_N$ of highest weight $\mu$ and $\lambda$ respectively, and where
the bar indicates the action of $\GL_N(\C)$ is composed with complex conjugation.
If $\MM$ corresponds to such a representation with $\lambda = 0$, then $\MM$ becomes locally
trivial up the $\p$-adic tower, and hence the coefficients can be pulled out as in Remark~\ref{remark:coefficients}. 
In particular, for such local systems, we have,
under the conditions of Theorem~\ref{theorem:exp},
isomorphisms 
$$H_n(\Gamma_N(\p^m),\MM) \simeq H_n(G_N(p^m),\MM \otimes_{\OL_F}  \Z_p).$$
}
\end{remark}

We note that (as it must) the algebra
$H^*(G_N,\Q_p)$  as $N \rightarrow \infty$ coincides with the stable cohomology
groups as computed by Borel.  The groups $H^*(G_N,\MM \otimes \Q_p)$  vanish
for non-trivial $\MM$ as can be seen by 
 considering the action of the infinitesimal character. This is consistent with the vanishing
 of 
$H^*(\Gamma_N,\MM \otimes \Q_p)$ in the stable range, as follows (essentially) from
a similar computation.

\medskip

We  observe that very regular primes do exist. 
For any particular fixed field $F$, we can give a relatively fast and explicit algorithm
for computing regular primes using generalized Bernoulli numbers. We give
some details in the case~$F = \Q(\sqrt{-1})$. Recall that the Bernoulli numbers $B_{2n}$
and Euler numbers $E_{2n}$ are defined by the following Taylor series:
$$\frac{t}{e^{t}-1} = \sum \frac{B_n}{n!}  \cdot t^n, \qquad
\frac{2}{e^t + e^{-t}} = \sum \frac{E_n}{n!} \cdot t^n.$$

\begin{lemma} \label{lemma:crit3}
If $F = \Q(\sqrt{-1})$, then $p \equiv 1 \mod 4$ is very regular if and only if
the following conditions are satisfied:
\begin{enumerate}
\item
 The prime $p$ divides neither $B_{2n}$ nor  $E_{2n}$ for $2n$ less than $p-1$.
\item If $p = a^2 + b^2$, then $(4ab)^{p-1} \not\equiv 1 \mod p^2$.
\end{enumerate}
\end{lemma}

\begin{proof} The first condition follows from the formulae
$$\zeta(1-2n) = - \frac{B_{2n}}{2n}, \qquad 
L(\chi_4,-2n) =  \frac{E_{2n}}{2},$$
the second is equivalent to the condition  $\alpha^{p-1} \not\equiv 1 \mod \alphabar^2$
where $\alpha = a + b i$.
\end{proof}

The 
non-regular primes $p$ 
less than $100$
that split in $\Q(\sqrt{-1})$ are
$p = 37$ (which divides $B_{32}$) and $p = 61$, which  divides:
$$\frac{-\pi^7}{6! \cdot 2^7} \cdot L(\chi_4,-6) = 
L(\chi_4,7) =  1 - \frac{1}{3^7} + \frac{1}{5^7} - \frac{1}{7^7} + \ldots
= 61 \cdot \frac{\pi^7}{6! \cdot 2^6}.$$
Using Lemmas~\ref{lemma:crit},~\ref{lemma:crit2}
and~\ref{lemma:crit3}, we
compute the following examples for $p=3$ and $|d_F| < 200$ as well as 
primes less than $100$ for the ten smallest imaginary quadratic fields.

\begin{example}  \emph{
 If  $F = \Q(\sqrt{-|d_F|})$, then $p = 3$ is very regular for a fundamental
discriminant $|d_F| < 200$ if  and only if $-d_F$ is one of the following integers:
$$8, 11,  20, 23, 59,  68, 71, 83, 95, 104, 116, 119, 131, 143, 152, 155, 167, 179, 191.$$
}
\end{example}

\begin{remark} \emph{Following~\cite{CL}, for an imaginary quadratic field~$F$, one may 
consider a heuristic model of the
$p$-class group  $\Cl(\OL_F) \otimes \Z_p$ as
 the quotient of a free $\Z_p$-module with $N$ generators by $N$ relations, 
where $N$ goes to infinity. Correspondingly,  one might model the ray class
group of conductor $(\pbar)^{\infty}$ for a $\p$ which is completely
split  by allowing an extra generator for ramification at $\pbar$, and
imposing a further relation by demanding local triviality at $\p$ (via the image of
the Artin map). In this case, one would predict that the probability
that $p$ fails to satisfy the given condition is exactly the same as the probability
that the $p$-part of the class group is trivial, namely:
$$\prod_{n=0}^{\infty} \left(1 - \frac{1}{p^n}\right).$$
For example, if $p = 3$, then  $p$ is very regular if and only if
it satisfies the final condition of Lemma~\ref{lemma:crit2},
 and this calculation suggests
that the density of such quadratic fields (amongst those
in which $3$ splits) is 
$$\prod_{n=0}^{\infty} \left(1 - \frac{1}{3^n}\right) \sim 0.560126\ldots$$
For such fields of fundamental discriminant $|d_F| <$ 1,000, 10,000, 
100,000, and 300,000, the corresponding percentage of fields in which
$3$ is very regular is 68.1\%, 64.6\%,  61.2\%, and 59.8\% respectively.
(This is consistent with the slow monotone rate of convergence
often observed in Cohen--Lenstra phenomena, and also with the
fact that the apparent error term is positive.) There   are natural maps:
$$\RCl((\pbar)^{\infty})/\p \otimes \Z_p \rightarrow \Cl(\OL_F)/\p \otimes \Z_p
\leftarrow \Cl(\OL_F) \otimes \Z_p.$$
From our discussion, the two outer groups should have a similar distribution. The
group in the middle, however, having one further relation, should behave like the $p$-part of
the class group of a \emph{real} quadratic field (this exact remark is made in
the last paragraph of~\S8 of~\cite{CL}).}
\end{remark}

\begin{example} 
For the $10$ smallest imaginary quadratic fields, we have the following table:
\begin{center}
\begin{tabular}{|c||c|c|c|c|c|c|c|c|c|c|}
\hline
$-d_F$ & $3$ & $4$ & $7$ & $8$ & $11$ & $15$ & $19$ & $20$ & $23$ & $24$ \\
\hline
$3$ &   & & & \checkmark & \checkmark  &  &  & \checkmark & \checkmark &  \\
$5$ &   & \checkmark &   &  & \ding{55} &   &  \checkmark &  &  &  \checkmark \\
$7$ & \checkmark &  &  &  &  &  & \checkmark & \checkmark &  & \checkmark \\
$11$ &  &  & \checkmark & \checkmark &  &  & $2$, \ding{55} &  &  & \checkmark \\
$13$ & \ding{55} & \checkmark &  &  &  &  &  &  &  \checkmark &  \\
$17$ &  & \checkmark &  & \checkmark &  & $14$ & \checkmark &  &  & \\
$19$ & \checkmark &   &  & $4$ &  & \checkmark &  & & & \\
$23$ &  &  &  \checkmark &  &  \checkmark & \checkmark & \checkmark & \checkmark & \checkmark & \\
$29$ &  &  \checkmark & $18$ &  &  &  &  &  $16$ & \checkmark & \checkmark \\
$31$ & \checkmark &  &   &  &  \checkmark & $4$ &  &  &  \checkmark & \checkmark \\
$37$ & $31$ & $31$ & $31$ & & $31$ &  & & & & \\
$41$ &  &  \checkmark &  &  \checkmark &   &  &  & $24$ & $8$  &  \\
$43$ & \checkmark & & \checkmark & \checkmark & & & \checkmark & \checkmark & & \\
$47$ & & & & & \checkmark & \checkmark & \checkmark & $24$ & \checkmark & \\
$53$ & &  \checkmark & $18,42$ &  & \checkmark & \checkmark &  &  & & \checkmark \\
$59$ &  & & & $18,43$ & $43$ & & & & $43$ & $43,52$ \\
$61$ &  \checkmark & $6$ &  & & & $54$ & $56$ & $42$ & & \\
$67$ &  $46,57$ &   & $57$ &   $57$ &   $57$ &   & & $57$ &   & \\
$71$ &  & & \checkmark &   & $6$ &   & & & $4$ &   \\
$73$ & \checkmark &  \checkmark &   & $30$ &   & & $58$ &   & \checkmark & $28$ \\
$79$ &  \checkmark &   & \checkmark &   & & \checkmark &   & & & \checkmark \\
$83$ &  & & &$14$ &   & \checkmark &   \checkmark &   $8, 24$ &   & \checkmark \\
$89$ & & \checkmark &   &   $32$ &  \checkmark &   & & \checkmark &   & \\
$97$ & \checkmark &   \checkmark &   & \checkmark &   \checkmark &   & & & & \checkmark \\
\hline
\end{tabular}
\end{center}
\noindent Here a tick indicates that $p$ is very regular for $F$,
a blank square indicates that $p$ is either inert or ramified in $F$, an even integer $2n$ means
that $p$ fails to be regular because $p$ divides $L(\chi,-2n)$, an odd
integer $2n-1$ means that $p$ divides $\zeta(1-2n)$, and the crosses indicate the failure of condition~$3$ of 
Lemma~\ref{lemma:crit2}, explicitly, $\alpha^{p-1} \equiv 1 \mod \alphabar^2$, where
$$\alpha = \frac{7 + \sqrt{-3}}{2}, \qquad \frac{3 + \sqrt{-11}}{2}, \qquad  \frac{5 + \sqrt{-19}}{2}$$
respectively --- note that $h_F = 1$ in each case.
\end{example}

It is interesting to compare the case $F = \Q(\sqrt{-2})$ and $3 = \p \pbar$ to
computations of \emph{non-stable} 
completed cohomology groups 
in the  work of the author  and Dunfield~\cite{CD} (see also~\cite{BE}).
In particular, for $N = 2$, one may show in this case
that  $\Htw_n(\p,\Z_p) = 0$ for all $n > 0$. Moreover, in the context 
of~\cite{CEtowers}, it is theoretically possible that $\Htw_n(\p,\Z_p) = 0$
for very regular primes for the same~$F$ and~$\p$  for other values of $N$, since one has the
necessary ``numerical coincidence'' between dimensions of locally symmetric spaces
and $p$-adic lie groups:
$$\dim \SL_N(\OL_F) \backslash \SL_N(\C) \slash \SU_N(\C)
=  N^2 - 1 = \dim \SL_N(\Z_p).$$
However, under the assumption that
for sufficiently large $N$ there will exist weight $0$ regular cuspidal automorphic
representations of level $1$ for $F = \Q$, one would anticipate 
the existence of non-stable characteristic zero classes,  
which would imply that $\Htw_*(\p,\Z_p)$ can
only  vanish over
the entire non-stable range only for finitely many $N$ (for all $F$ and $\p$).

\begin{remark}  \emph{
Natural heuristics suggest that for any imaginary quadratic field $F$, there
are infinitely many very regular primes (possibly with density $e^{-1} =  0.367879\ldots$). 
This seems hard to prove, however, since being very
regular  implies that $p$ is regular in $\Q$, and the infinitude of regular primes is a well
known open question. 
}
\end{remark}

\begin{remark} \emph{There is  a natural generalization of very regular to a prime $p$
that splits completely in a totally imaginary CM field $F/F^{+}$. Here one replaces
a single place $v|p$ by a collection $T$ of $[F^{+}:\Q]$ places $v|p$,
and asks that the map:
$$\prod_{[F^{+}:\Q]} \Gal(\Qbar_p/\Q_p) \rightarrow \prod_{v|T} D_v \subseteq \Gamma$$
is surjective, where $\Gamma$ is the Galois group of the pro-$p$ extension of
$F(\zeta_p)$ unramified outside~$p$.
}
\end{remark}

\begin{remark} \emph{In contrast to the case of regular primes, it seems
difficult to control  the module $\Htw_d(\SL,\p,\Z_p)$ for a prime $p$
that splits in $F$, even for $d = 2$. An elementary computation
 shows that $\Htw_2(\SL,\p,\Z_p)$ is infinite
if and only if $L_p(\chi,2) \ne 0$. (Wagoner computed that
the associated regulator map was non-zero for~$F = \Q(\sqrt{-3})$ and primes~$\p$
of prime norm $\le 73$ in~\cite{Wagonerc}.) As with $\zeta_p(3) \in  \Q_p$,
it seems difficult to prove that $L_p(\chi,2) \ne 0$ in general, even for a fixed $p$ (and varying
 imaginary quadratic character $\chi$).
    }
    \end{remark}

\bibliographystyle{amsalpha}
\bibliography{CalegariStable}

\end{document}